\documentclass[11pt]{amsart}

\usepackage{ytableau}

\usepackage{amsmath, amsthm, amssymb,mathtools,bm,amsfonts}
\usepackage{tikz-cd}
\usepackage{tikz}
\usepackage[T1]{fontenc}

\usepackage[pdftex,hidelinks,backref=page]{hyperref}
\hypersetup{
    colorlinks,
    citecolor=magenta,
    filecolor=magenta,
    linkcolor=blue,
    urlcolor=black
}

\usepackage{xcolor}

\renewcommand{\emptyset}{\varnothing}
\newcommand{\rank}{\operatorname{rank}}

\theoremstyle{definition}
\newtheorem{thm}{Theorem}[section]
\newtheorem{cor}[thm]{Corollary}
\newtheorem{lem}[thm]{Lemma}

\newtheorem{prop}[thm]{Proposition}
\newtheorem{defn}[thm]{Definition}

\newtheorem{eg}[thm]{Example}
\newtheorem{rem}[thm]{Remark}

\newtheorem{question}[thm]{Question}
	\newtheorem{fact}[thm]{Fact}

\numberwithin{equation}{section}

\newcommand{\ar}[1]
{{\xrightarrow{#1}}}

\newcommand{\sym}[1]{\operatorname{Sym}_{#1}} 
\newcommand{\suchthat}{\;|\;}

\newcommand{\schub}[1]{\mathfrak{S}_{#1}} 
\date{}

\newcommand{\idem}{\operatorname{id}} 
\newcommand{\ct}{\operatorname{ev_0}} 

\newcommand{\fl}[1]{\mathrm{Fl}_{#1}}
\newcommand{\GL}{\mathrm{GL}}

\renewcommand\emph[1]{\textcolor{blue}{\textit{#1}}} 

\newcommand{\syt}{\operatorname{SYT}} 
\newcommand{\gl}{\operatorname{GL}}

\newcommand{\XM}{\mathcal{B}_{\mathcal{M}}}
\newcommand{\sgn}{\operatorname{sgn}}

\title[Two-row Springer fiber components]{The Springer geometry of Specht polynomials, and Schubert cycle positivity for two-row Springer fiber components}

\author{Hunter Spink}
\address{Department of Mathematics,
University of Toronto, Toronto, ON M5S 2E4, Canada}
\email{\href{mailto:hunter.spink@utoronto.ca}{hunter.spink@utoronto.ca}}

\author{Vasu Tewari}
\address{Department of Mathematical and Computational Sciences, University of Toronto Mississauga, Mississauga, ON L5L 1C6, Canada}
\email{\href{mailto:vasu.tewari@utoronto.ca}{vasu.tewari@utoronto.ca}}

\thanks{
HS and VT acknowledge the support of the NSERC, respectively [RGPIN-2024-04181] and [RGPIN-2024-05433].}

\begin{document}

\begin{abstract}
We show that the type $A$ Springer representation is realized geometrically in the
homology of the complete flag variety by Specht polynomials.
For any partition
$\lambda$, we identify the classical Specht polynomial generators of the Specht module
$V_\lambda$ with the classes of a family of disjoint \textit{Levi--Richardson} varieties, and this family degenerates to the Springer fiber $\mathcal{B}_\lambda$.
This factors Springer's Schubert positivity problem for Springer fiber components through a chain of positive
expansions, from Specht polynomials through the Joseph polynomials to the Schubert
cycles.

For two-row partitions we make each of these expansions combinatorially
explicit, giving manifestly nonnegative Schubert cycle expansions of both the
Levi-Richardson cycles and the Springer fiber components. This resolves Springer's
question for two-row fibers and proves two conjectures of Precup and Sabando-Alvarez, and
identifies the Springer basis with the web basis for two-row Specht modules. As an
application, we deduce the Schubert cycle expansions of the components of the Poisson
degeneracy locus of the flag variety.
\end{abstract}

\maketitle

\section{Introduction}
We work over $\mathbb{C}$. Let $B,B^-\subset \gl_n$ denote the upper and lower triangular invertible matrices. The complete flag variety $\fl{n}\coloneqq \gl_n/B$ is identified with the set of flags of subspaces $$\fl{n}=\{\{0\}=V_0\subsetneq V_1\subsetneq V_2\subsetneq \cdots \subsetneq V_{n-1}\subsetneq V_n=\mathbb{C}^n\suchthat \dim V_i=i\}$$
via the transitive action of $\gl_n$ with $B$ the stabilizer of the coordinate flag $V_i=\langle {\sf e}_1,\ldots,{\sf e}_i\rangle$. We denote the set of complete flags preserved by an $n\times n$ matrix $M$ by
$$\mathcal{B}_M\coloneqq \{(V_i)_{i=0}^{n}\suchthat MV_i\subset V_i\}\subset \fl{n}.$$
Our primary interest is in the \emph{Springer fiber}
$\mathcal{B}_\lambda\coloneqq \mathcal{B}_N\subset \fl{n}$, where $N$ is a fixed nilpotent matrix of type $\lambda=(\lambda_1\ge \lambda_2\ge \cdots \ge \lambda_\ell)$.
As shown by Spaltenstein \cite{Spa76}, the Springer fiber decomposes into irreducible components $\mathcal{B}_T$ indexed by standard Young tableaux (SYTs):
$$\mathcal{B}_\lambda=\bigcup_{T\in \syt(\lambda)}\mathcal{B}_T.$$
It is equidimensional of dimension $\dim \mathcal{B}_\lambda=\sum_{i}^{\lambda_1} \binom{\lambda_i^{\top}}{2}$, where $\lambda^{\top}$ is the transposed partition.
Here and throughout, an unadorned $\sum\binom{\lambda_i^{\top}}{2}$ denotes this sum over $1\le i\le \lambda_1$.

The homology of the flag variety carries the \emph{Springer action} of the symmetric group $S_n$, realizing the regular representation $H_\bullet(\fl{n})\cong \mathbb{Z}[S_n]$ with a distinguished geometric basis given by the \emph{Schubert cycles} $X^w\coloneqq \overline{BwB}$, $w\in S_n$.
The action descends to $H_{\dim \mathcal{B}_\lambda}(\mathcal{B}_\lambda)=\bigoplus_{T\in \syt(\lambda)}\mathbb{Z}[\mathcal{B}_T]$, giving the \emph{Springer representation}: a copy of the Specht module $V_\lambda\subset H_\bullet(\fl{n})$, with the classes $\{[\mathcal{B}_T]\}$ forming the \emph{Springer basis}.
Expressing the Springer basis in the Schubert basis is a long-standing question of Springer \cite{springerconj}, open beyond a few special cases \cite{Gu89,spink2025richardsontableauxschubertpositivity}. We resolve it for all \emph{two-row} Springer fibers, $\lambda=(n-k,k)$, and along the way prove two conjectures of Precup--Sabando-Alvarez \cite[Conjectures 1,2]{sabandoalvarez2026idealsdefiningcomponentstworow}.

We show that Springer's problem is a special case of a broader one concerning the classical \emph{Specht polynomial} generating set for $V_\lambda$.
Our route runs through a degeneration communicated to us by  Martha Precup \cite{MarthaPriv},
which at one end is a disjoint union of varieties we call $\lambda^{\top}$-Levi Richardsons and at the other end is $\mathcal{B}_\lambda$ itself.
One of our main theorems identifies the $\lambda^{\top}$-Levi Richardsons with the Specht polynomial generators of $V_\lambda$, and the degeneration yields the first of the following series of positive expansions in $H_\bullet(\fl{n})$:
$$\underbrace{\lambda^{\top}\text{-Levi-Richardsons}}_{\substack{\text{Specht polynomials $f_T$}\\\text{(Section~\ref{sec:SpringerRepRichardson})}}}\xrightarrow{\text{positively expands}}\underbrace{\text{Springer basis elements}}_{\substack{\text{Joseph polynomials $J_T$}\\\text{(Section~\ref{sec:Joseph}})}}\xrightarrow{\text{positively expands}}\underbrace{\text{Schubert cycles}}_{\substack{\text{Schubert degree}\\\text{polynomials } \mathcal{D}_w\\\text{(Section~\ref{subsec:divdiff})}}},$$
where the Joseph polynomials are the least understood.
Combinatorializing these expansions would be very interesting (see Questions~\ref{question:RctoSchubert} and ~\ref{question:PositiveExpansions}).
Before turning to the two-row case, we orient the reader to the three families of
polynomials in this diagram and to what is already known about the two expansions.

\begin{enumerate}
    \item \textit{The polynomial model.}
    To work with polynomial representatives for homology classes, we embed $H_\bullet(\fl{n})$ in the ring of root polynomials $\mathbb{Z}[\{z_i-z_j\}_{1\le i < j \le n}]$ by an integral normalization of the symmetric harmonics of Bernstein--Gelfand--Gelfand \cite{BGG73}.
    We note that effective classes always expand into \emph{positive root polynomials} $\mathbb{N}[\{z_i-z_j\}_{1\le i < j \le n}]$.
    \item \textit{The rightmost family.} The Schubert degree polynomials $\mathcal{D}_w(z_1,\ldots,z_n)$ appear implicitly in \cite[Theorem 3.13]{BGG73}. They are the degree polynomials $\mathcal{D}_w^{PS}(y_1,\ldots,y_{n-1})$ studied in Postnikov--Stanley \cite{PS09} considered in the integral normalization, i.e. after applying $\prod y_i^{a_i}\mapsto \prod a_i!(z_i-z_{i+1})^{a_i}$. Every effective cycle expands nonnegatively into these degree polynomials, and as shown in \textit{ibid.} the expansion into $\mathbb{N}[\{z_i-z_j\}_{1\le i < j \le n}]$ is encoded by certain chains in the Bruhat order.
    \item \textit{The middle family.} The \emph{Joseph polynomials} \cite{Jo84,Jo89,KZJ14, RTVZJ12} are usually defined as multidegrees of orbital varieties. Hotta \cite{Ho84} showed that the span of the Joseph polynomials for $\lambda$ is isomorphic, as an $S_n$-representation, to the Springer representation, with $J_T$ corresponding to $[\mathcal{B}_T]$; we show in Section~\ref{sec:Joseph} that with our normalizations the two agree identically.
    \item \textit{The second arrow.}
    The expansion of Joseph polynomials into
    $\mathbb{N}[\{z_i-z_j\}_{i<j}]$ was considered in \cite[Section 8]{Jo89}; it would
    be interesting to obtain such combinatorial expansions directly, should decomposing
    them into the coarser Schubert degree polynomials prove too difficult.
    \item \textit{Prior work for two columns.} For $\lambda_1\le 2$ the complete flag variety is spherical with respect to the $\lambda^{\top}$-Levi subgroup of $\GL_n$. The Schubert cycle expansion of the
    $\lambda^{\top}$-Levi Richardsons was computed by Wyser \cite{W13}, and the Joseph
    polynomials via orbital varieties by Knutson--Zinn-Justin \cite{KZJ14} and by Rimanyi--Tarasov--Varchenko--Zinn-Justin  \cite{RTVZJ12}. The intermediate positive
    expansions are not known combinatorially.
\end{enumerate}

For two-row $\lambda=(n-k,k)$ we give combinatorial rules for these expansions.
The first expansion is the expansion of the Specht polynomials into the web basis via skein relations \cite{Rho19}.
The second is more delicate, relying on a fortuitous interaction between long-range divided differences in type $A$ and setting variables to zero, studied by the present authors with Nantel Bergeron, Lucas Gagnon, and Philippe Nadeau \cite{BGNST1,NST_c}, with the combinatorics being governed by the geometry of the \emph{quasisymmetric flag variety} $\operatorname{QFL}_n\subset \fl{n}$ \cite{BGNST2}.
We do not know a direct geometric link between $\operatorname{QFL}_n$ and two-row Springer fiber components, yet the techniques developed for Schubert positivity in the former proved indispensable for the latter.

As a final application we show that the Poisson degeneracy locus of $\fl{n}$, as described by Casbi--Masoomi--Yakimov~\cite{CMY}, has components which are left $S_n$-translates of the two-row $\lambda^{\top}$-Levi Richardsons, and therefore our results also compute their Schubert cycle expansions.

\subsection{The degeneration}
Fix $\lambda$ a partition of $n$.
For a diagonal matrix $D_{\lambda^{\top}}$ with eigenspaces of dimensions $\lambda^{\top}$ there is a decomposition $$
\mathcal{B}_{D_{\lambda^{\top}}}=\bigsqcup_{\mathcal{C}\in S_{\lambda^{\top}}\setminus S_n} \mathcal{R}_{\mathcal{C}}
$$
indexed by right cosets of the parabolic subgroup $S_{\lambda^{\top}}\coloneqq \prod S_{\lambda_i^{\top}}\subset S_n$, with each $\mathcal{R}_\mathcal{C}$ irreducible of dimension $\sum \binom{\lambda_i^{\top}}{2}=\dim \mathcal{B}_\lambda$.
These components admit several equivalent descriptions:
\begin{enumerate}
    \item as Levi orbits $\mathcal{R}_{\mathcal{C}}=\GL_{\lambda^{\top}}wB$ for any $w\in \mathcal{C}$, and Levi subgroup $\GL_{\lambda^{\top}}\coloneqq\prod \GL_{\lambda_i^{\top}}\subset \GL_n$;
    \item as Richardson varieties $\mathcal{R}_{u,v}\coloneqq \overline{BvB}\cap \overline{B^-uB}$, for $u,v$ the minimal and maximal length representatives in $\mathcal{C}$;
    \item as embedded copies of $\prod \fl{|\mathcal{P}_i|}\hookrightarrow \fl{n}$ under the pattern map of Bergeron--Sottile \cite{BS98,BS99} associated to ordered set partitions $\mathcal{P}_1\sqcup \cdots \sqcup \mathcal{P}_{\lambda_1}= \{1,\ldots,n\}$ with $|\mathcal{P}_i|=\lambda_i^{\top}$.
\end{enumerate}

We will henceforth refer to these as \emph{$\lambda^{\top}$-Levi Richardsons}.
The degeneration we are concerned with takes $\mathcal{B}_{D_{\lambda^{\top}}}\to \mathcal{B}_\lambda$ by degenerating $D_{\lambda^{\top}}$ to the nilpotent matrix $N$ of Jordan type $\lambda$ while staying within the set of diagonalizable matrices with eigenspaces of dimensions $\lambda^{\top}$ (see Section~\ref{sec:degeneration}).

\subsection{The Springer geometry of Specht polynomials}

The homology of the flag variety can be understood as an $S_n$-representation via the embedding $H_\bullet(\fl{n})\hookrightarrow \mathbb{Z}[z_1,\ldots,z_n]$ given by associating to a homology class $[X]$ the \emph{degree polynomial}
\begin{equation}
\label{eqn:generating}\mathcal{D}_X(z_1,\ldots,z_n)\coloneqq \sum_{a_1+\cdots+a_n=\dim(X)}\left(\int_{X}x_1^{a_1}\cdots x_n^{a_n}\right)z_1^{a_1}\cdots z_n^{a_n}\in \mathbb{Z}[z_1,\ldots,z_n]\end{equation}
where $x_i=-c_1(\mathcal{F}_i/\mathcal{F}_{i-1})$ is the negative $i$'th Chern root of the tautological flag of subbundles $0\subsetneq \mathcal{F}_1\subsetneq \cdots \subsetneq \mathcal{F}_{n-1}\subsetneq \underline{\mathbb{C}}^n$, and the Springer action of $S_n$ on $H_\bullet(\fl{n})$ comes from $S_n$ permuting the $z_i$ variables. Because $x_1+\cdots+x_n=0$ in $H^\bullet(\fl{n})$, the image of $H_\bullet(\fl{n})$ factors through the \emph{root polynomials} $\mathbb{Z}[z_1-z_2,z_2-z_3,\ldots,z_{n-1}-z_n]$.

\begin{thm}[{Section~\ref{sec:SpringerRepRichardson}}]
\label{thm:introSpecht}
Fix a partition $\lambda$ of $n$.
\begin{enumerate}
    \item The Springer representation is the Specht polynomial realization of $V_\lambda$, i.e. generated by the \emph{Specht polynomials} $$f_T\coloneqq \prod_{a\text{ above $b$ in the same column of }T} (z_a-z_b)\in \mathbb{Z}[z_1,\ldots,z_n]$$
    for column-strict fillings $T$ of $\lambda$ with distinct entries $\{1,\ldots,n\}$.
    \item The Specht polynomials $f_T$ are the degree polynomials of the $\lambda^{\top}$-Levi Richardsons  $\mathcal{R}_{\mathcal{C}}$.
    More specifically there is a natural bijection  $\mathcal{C}\mapsto T(\mathcal{C})$ from $S_{\lambda^{\top}}\setminus S_n\to $ column-strict fillings of $\lambda$, and we have
    $$\mathcal{D}_{\mathcal{R}_{\mathcal{C}}}(z_1,\ldots,z_n)=f_{T(\mathcal{C})}.$$
\end{enumerate}
In particular the Springer representation is spanned over the integers by the $\lambda^{\top}$-Levi Richardsons:
\begin{equation}
\label{eqn:lusztigeqn}
\bigoplus \mathbb{Z}[\mathcal{B}_T]=\mathbb{Z}\{[\mathcal{R}_{\mathcal{C}}]\suchthat \mathcal{C}\in S_{\lambda^{\top}}\setminus S_n\}\subset H_\bullet(\fl{n}).\end{equation}
\end{thm}

\begin{rem}
Various aspects of this theorem were known previously, though not in the present form.
    \begin{enumerate}
        \item[(a)] Over $\mathbb{Q}$, under the variable substitution $z_1^{a_1}\cdots z_n^{a_n}\mapsto \frac{1}{a_1!a_2!\cdots a_n!}\prod y_i^{a_i}$ the image of $H_\bullet(\fl{n})$  is the space of symmetric harmonics \cite[Theorem~3.13]{BGG73}, those rational polynomials killed by every differential operator $f(\frac{d}{dy_1},\ldots,\frac{d}{dy_n})$ with $f$ symmetric and without constant term.
        Specht polynomials in the $z_i$ variables map to scalar multiples of Specht polynomials in the $y_i$ variables, so the combinatorial content of part (1) is that these Specht polynomials lie in the harmonic space of the Tanisaki ideal $\mathcal{I}_\lambda$, and follows from work of Rhoades--Yu--Zhao \cite[Lemma 11]{RYZ20} -- we will give a short direct proof of the special case we need.
        \item[(b)]  Martha Precup \cite{MarthaPriv} communicated to us that the degeneration considered above yields a geometric proof of \eqref{eqn:lusztigeqn} over the rationals.
        By identifying the Specht polynomials explicitly we can upgrade this to work over $\mathbb{Z}$ (Theorem~\ref{thm:Zupgrade}).
    \end{enumerate}
\end{rem}

\subsection{Combinatorial nonnegativity conjectures}
For geometric reasons \cite{Bri05} the fundamental cycle of any subvariety $[X]\in H_\bullet(\fl{n})$ expands nonnegatively in the Schubert cycle basis
$$[X]=\sum_{w\in S_n} h^X_w[X^w]\text{ with }h^X_w\in \mathbb{N}.$$
For $X=X^w$, we denote the associated \emph{Schubert degree polynomial} $\mathcal{D}_w(z_1,\ldots,z_n)$.
The nonnegative degree polynomials $\mathcal{D}_w^{PS}(y_1,\ldots,y_{n-1})$ of Postnikov--Stanley \cite{PS09} recover $\mathcal{D}_w(z_1,\ldots,z_n)$ under the substitution $y_1^{a_1}\cdots y_{n-1}^{a_{n-1}}\mapsto \prod_{i=1}^{n-1} a_i!(z_i-z_{i+1})^{a_i}$.
Then the above expansion is identical to
$$\mathcal{D}_X(z_1,\ldots,z_n)=\sum_{w\in S_n}h^X_w\mathcal{D}_w(z_1,\ldots,z_n)\text{ with }h_w^X\in \mathbb{N}.$$
For geometrically natural $X$, a central question in algebraic combinatorics is to find manifestly nonnegative combinatorial rules for $h_w^X$.
\begin{rem}
 The coefficient of $(z_1-z_2)^{a_1}\cdots(z_{n-1}-z_n)^{a_{n-1}}$ in $\mathcal{D}_X(z_1,\ldots,z_n)$ is
$$\int_X x_1^{a_1}(x_1+x_2)^{a_2}\cdots (x_1+\cdots+x_{n-1})^{a_{n-1}}\in \mathbb{N}$$
which is nonnegative as $x_1+\cdots+x_i$ is the first Chern class of the ample line bundle $\det \mathcal{F}_i^{\vee}$.
Therefore effective cycles in $H_\bullet(\fl{n})$ map to positive root polynomials $\mathbb{N}[z_1-z_2,z_2-z_3,\ldots,z_{n-1}-z_n]$.
Postnikov--Stanley \cite{PS09} give a combinatorial rule via chains in the Bruhat order for the coefficients of $\mathcal{D}^{PS}_w(y_1,\ldots,y_n)$, and therefore we can deduce a nonnegative expansion of $\mathcal{D}_X(z_1,\ldots,z_n)$ into positive root polynomials from a nonnegative expansion into Schubert degree polynomials.
\end{rem}

For $X=\mathcal{R}_{\mathcal{C}}$, finding such a rule for $h^X_w$ is equivalent to a Schubert positivity question of the alphabet coproduct map.
The Borel presentation \cite{Bor53} of the cohomology ring of the flag variety is as the ring of symmetric coinvariants
$$
H^\bullet(\fl{n})=\mathbb{Z}[x_1,\ldots,x_n]/\langle f-\ct f\suchthat f\in \mathbb{Z}[x_1,\ldots,x_n]^{S_n}\rangle.
$$
This has a distinguished basis of the \emph{Schubert polynomials} $\{\schub{w}(x_1,\ldots,x_n)\suchthat w\in S_n\}$ of Lascoux--Sch\"utzenberger \cite{LS82}, which are the Kronecker dual basis in cohomology to the Schubert cycle basis under the natural perfect pairing $H^\bullet(\fl{n})\otimes H_\bullet(\fl{n})\to \mathbb{Z}$.

There is great interest in the combinatorial nonnegativity of $h^X_w$ when $X=\mathcal{R}_{u,v}$ is a Richardson variety, as then $h^X_w$ is the coefficient of $\schub{v}$ in the product $\schub{u}\schub{w}$, so they are the \emph{generalized Littlewood--Richardson coefficients}.
In the special case $\mathcal{R}_{\mathcal{C}}$ for $\mathcal{C}\in S_{\lambda^{\top}}\setminus S_n$ they have an alternative combinatorial interpretation in terms of the alphabet coproduct because of their connection to the pattern maps of Bergeron--Sottile.

More precisely, given an ordered set partition $\mathcal{P}_1\sqcup \cdots \sqcup \mathcal{P}_{r}=\{1,\ldots,n\}$, Bergeron--Sottile \cite{BS98,BS99} showed that every coefficient in the expansion
$$\schub{w}(x_1,\ldots,x_n)=\sum_{(w_1,\ldots,w_r)\in S_{|\mathcal{P}_1|}\times \cdots \times S_{|\mathcal{P}_r|}}a_{w_1,\ldots,w_r}\prod_{i=1}^r\schub{w_i}(\{x_j\}_{j\in \mathcal{P}_i})\in \bigotimes_{i=1}^r H^\bullet(\fl{|\mathcal{P}_i|})$$
can be expressed positively in terms of generalized Littlewood--Richardson coefficients.
The coefficient where $w_1,\ldots,w_r$ are all the longest elements of their respective symmetric groups is the coefficient $h^X_w$ where $X=\prod \fl{|\mathcal{P}_i|}\hookrightarrow \fl{n}$ is embedded via the pattern mapping.

\begin{question}
\label{question:RctoSchubert}
    For $\mathcal{C}\in S_{\lambda^{\top}}\setminus S_n$, find a combinatorially nonnegative algorithm for the coefficients of the expansion $[\mathcal{R}_{\mathcal{C}}]=\sum h^{\mathcal{R}_\mathcal{C}}_w[X^w]$, or equivalently of the Specht polynomial $f_{T(\mathcal{C})}(z_1,\ldots,z_n)=\sum h_w^{\mathcal{R}_{\mathcal{C}}}\mathcal{D}_w(z_1,\ldots,z_n)$ into Schubert degree polynomials.
\end{question}

The geometric degeneration gives a potential avenue for resolving this question.
The existence of the degeneration $\mathcal{B}_{D_{\lambda^{\top}}}\to \mathcal{B}_\lambda$ directly implies
$$[\mathcal{R}_{\mathcal{C}}]=\sum g^{\mathcal{C}}_T [\mathcal{B}_T]\text{ with }g^{\mathcal{C}}_T\ge 0,$$
so we can factor the above question into two potentially easier questions.
We will show later that under our normalizations  $$\mathcal{D}_{\mathcal{B}_T}(z_1,\ldots,z_n)=J_T(z_1,\ldots,z_n),$$ where $J_T(z_1,\dots,z_n)$ is the \emph{Joseph polynomial}  \cite{Jo84,Jo89,KZJ14, RTVZJ12}, defined as the multidegree of an orbital variety (see Section~\ref{sec:Joseph}).
\begin{question}
\label{question:PositiveExpansions}\leavevmode
\begin{enumerate}
    \item\label{it:RtoB} Find a combinatorially nonnegative algorithm for the coefficients $g^{\mathcal{C}}_T$ above, i.e. show combinatorially that the Specht polynomials expand nonnegatively into Joseph polynomials.
    \item\label{it:SpringerQuestion} (Springer \cite{springerconj}) Find a combinatorially nonnegative algorithm for the coefficients $h_w^{\mathcal{B}_T}$ above, i.e. show that the Joseph polynomials expand nonnegatively into degree polynomials .
\end{enumerate}
\end{question}

A combinatorial formula for $h^{\mathcal{B}_T}_w$ was determined by G\"uemes \cite{Gu89} for $\lambda=(n-k,1^k)$ of hook shape. The tableaux in this case are all ``Richardson tableaux'' of Karp--Precup \cite{karp2025richardsontableauxcomponentsspringer}, meaning there is an equality with a Richardson variety $\mathcal{B}_T=\mathcal{R}_{u,v}$, and a combinatorial formula for $h^{\mathcal{B}_T}_w$ for all Richardson tableaux $T$ was recently established by the present authors \cite{spink2025richardsontableauxschubertpositivity}.

\subsection{Two-row Springer fiber components}

We carry out the above program in full in the context of \emph{two-row Springer fibers} \cite{Fun03,SW12,GNST25}, i.e. with $\lambda=(n-k,k)$. Each right coset  $\mathcal{C}\in S_{\lambda^{\top}}\setminus S_n$ canonically induces a matching $\mathcal{M}$ of $\{1,\dots,n\}$, and we then have
$$\mathcal{D}_{\mathcal{R}_{\mathcal{C}}}(z_1,\ldots,z_n)=\prod_{(i<j)\in \mathcal{M}}(z_i-z_j)\in H_\bullet(\fl{n}).$$
As described in Fung \cite{Fun03} and Stroppel--Webster \cite{SW12}, there is a natural way to associate what we will call a ``packed noncrossing matching'' of $[n]$ to a tableau $T\in \syt(\lambda)$, so we may also write $\mathcal{B}_{\mathcal{M}}$ for the component associated to $T$. Then we show $$\mathcal{D}_{\mathcal{B}_{\mathcal{M}}}(z_1,\ldots,z_n)=\prod_{(i<j)\in \mathcal{M}}(z_i-z_j)\in H_\bullet(\fl{n})$$
as well, so these arise as a distinguished subset of the Levi-Richardson homology classes (the corresponding statement for Joseph polynomials was known previously, see Remark~\ref{rem:AlternateJoseph} and \cite{RTVZJ12}).
This identifies the Springer basis/Joseph polynomial basis $\{[\mathcal{B}_\mathcal{M}]\}$ for $V_\lambda$ with the \emph{web basis} associated to packed noncrossing matchings.
That the Specht-to-web transition is unitriangular is due to Russell--Tymoczko \cite{RT19}; what is new here is the geometric identification of the web basis with the Springer basis of two-row fiber components, via degree polynomials.
There are positive straightening rules \cite{Rho19} for expanding Specht polynomials into the web basis, and these yield combinatorially positive rules for the expansions of these Levi-Richardson cycles into the two-row Springer fiber cycles
$$[\mathcal{R}_{\mathcal{C}}]=\sum_{\text{packed noncrossing matchings } \mathcal{M}} g^{\mathcal{M}}_{\mathcal{C}}[\mathcal{B}_{\mathcal{M}}]\text{ with }g_{\mathcal{C}}^{\mathcal{M}}\in \mathbb{N}.$$
For two-row $\lambda$ this resolves Question~\ref{question:PositiveExpansions}\eqref{it:RtoB}.

We also resolve Springer's question (Question~\ref{question:PositiveExpansions}\eqref{it:SpringerQuestion}) for all components $\mathcal{B}_T$ of two-row Springer fibers, and additionally compute the reverse Artin representatives of their Poincar\'e dual cohomology classes. This verifies two conjectures of Precup and Sabando-Alvarez \cite[Conjectures 1,2]{sabandoalvarez2026idealsdefiningcomponentstworow}.
We therefore have in $H_\bullet(\fl{n})$ combinatorial expansions for $\lambda=(n-k,k)$:
$$\{[\mathcal{R}_{\mathcal{C}}]\suchthat \mathcal{C}\in S_{\lambda^{\top}}\setminus S_n\}\xrightarrow{\text{positively expands}}\{[\mathcal{B}_{T}]\suchthat T\in \syt(\lambda)\} \xrightarrow{\text{positively expands}}\{[X^w]\suchthat w\in S_n\},$$
and hence fully resolve Question~\ref{question:RctoSchubert} in this special case via the program outlined above.

Our strategy follows \cite{spink2025richardsontableauxschubertpositivity},
which expressed the degree map on $\mathcal{B}_T$ for $T$ a Richardson tableau as a composite of long-range divided differences $\partial_{ij}$ interspersed with setting variables to zero.
Using the Stroppel--Webster presentation of the cohomology ring of two-row Springer fiber components \cite[Theorem 9]{SW12}, we establish such a form for the degree map on $\mathcal{B}_T$ in the two-row case.
We do not have a geometric reason to expect that  the component homology classes in the two-row and Richardson tableau cases would be so expressible -- they seem independently fortuitous.

The fact that long-range divided difference operations satisfy combinatorially positive straightening relations into products of short-range divided differences when certain variables are set to zero is the key insight in the formalism of quasisymmetric divided differences of the present authors and Philippe Nadeau \cite{NST_a}, and motivated the techniques of both \cite{spink2025richardsontableauxschubertpositivity} and the present paper.

Finally, in Section~\ref{sec:FurtherMatchings} we relate our varieties to two others indexed by matchings. Casbi--Masoomi--Yakimov \cite{CMY} study Richardson varieties $\mathcal{R}_{v,w}^{CMY}$ -- components of the Poisson degeneracy locus of $\fl{n}$ -- indexed by a matching together with extra data; and among the subvarieties of $\operatorname{QFL}_n$ are left-translates of Richardson varieties $X(\mathcal{M})$ indexed by noncrossing matchings. In both cases these are left $S_n$-translates of $\mathcal{R}_{\mathcal{C}}$ varieties, so our results yield combinatorially nonnegative Schubert cycle expansions for $[\mathcal{R}_{v,w}^{CMY}]$ and $[X(\mathcal{M})]$ -- recovering, for $X(\mathcal{M})$, the expansion of \cite{NST_c,BGNST2}.

\subsection*{Acknowledgements}
We would like to thank Martha Precup and Cristina Sabando-Alvarez for helpful conversations. We are extremely grateful in particular to Martha for informing us of the existence of \eqref{eqn:lusztigeqn}, and Ruizhen Liu for supplying a significant simplification to the presentation of the degeneration in Section~\ref{sec:degeneration}.

\section{Preliminaries}
Write $[n]=\{1,\ldots,n\}$ and let $S_n$ denote the symmetric group of permutations of $[n]$, with simple transpositions $s_i=(i,i+1)$, reflections $\tau_{ij}=(i,j)$, and longest word $w_\circ\in S_n$. In what follows $S_n$ acts on any polynomial ring by permuting variables, $\sigma\cdot f(y_1,\ldots,y_n)=f(y_{\sigma(1)},\ldots,y_{\sigma(n)})$. For $D\subset[n]$ we denote by $Y_D$ the operation
$$Y_Df(x_1,\ldots,x_n)=f(x_1,\ldots,x_n)|_{x_i=0\text{ for }i\in D},$$
and write $\ct\coloneqq Y_{[n]}$ for the constant term map $f\mapsto f(0,0,\ldots,0)$.

\subsection{Divided differences, Schubert polynomials, and degree polynomials}
\label{subsec:divdiff}

The flag variety $\fl{n}$ is equipped with the natural tautological flag of subbundles $$\{0\}=\mathcal{F}_0\subsetneq \mathcal{F}_1\subsetneq \cdots \subsetneq \mathcal{F}_{n-1}\subsetneq \mathcal{F}_n=\underline{\mathbb{C}}^n,$$
where the fiber of $\mathcal{F}_i$ over $MB$ is the span $V_i$ of the first $i$ columns of $M$. Denoting $x_i=-c_1(\mathcal{F}_i/\mathcal{F}_{i-1})$ and $\sym{n}\coloneqq \mathbb{Z}[x_1,\ldots,x_n]^{S_n}$,
the cohomology of $\fl{n}$ has the well-known Borel presentation \cite{Bor53}
$$H^\bullet(\fl{n})=\mathbb{Z}[x_1,\ldots,x_n]/\langle f-\ct f\suchthat f\in \sym{n}\rangle$$
Define the \emph{divided difference operation} $\partial_i=\frac{\idem-s_i}{x_i-x_{i+1}}$.
These satisfy the nil-Hecke relations
$$\partial_i^2=0\quad \partial_i\partial_{i+1}\partial_i=\partial_{i+1}\partial_i\partial_{i+1}\quad \partial_i\partial_j=\partial_j\partial_i\text{ for }|i-j|\ge 2.$$

These relations imply that $\partial_{i_1}\cdots \partial_{i_k}=0$ if $s_{i_1}\cdots s_{i_k}=w$ is not a reduced word, and is equal to a fixed operator $\partial_w$ independent of the choice of reduced word otherwise.
The Schubert polynomials of Lascoux--Sch\"utzenberger \cite{LS82} $\{\schub{w}\suchthat w\in S_n\}\subset H^\bullet(\fl{n})$ are a $\mathbb{Z}$-basis for the cohomology of $\fl{n}$.
They are recursively defined by setting $\schub{w_\circ}=x_1^{n-1}x_2^{n-2}\cdots x_{n-1}^1$ and $\partial_i\schub{w}=\delta_{w(i)>w(i+1)}\schub{ws_i}$.
The divided differences $\partial_1,\ldots,\partial_{n-1}$ descend to endomorphisms of $H^\bullet(\fl{n})$, and the Schubert polynomials satisfy the duality \begin{equation}
\label{eqn:ev0duality}\ct\partial_{w}\schub{w'}(x_1,\ldots,x_n)=\begin{cases}1&w=w'\\0&\text{otherwise.}\end{cases}\end{equation}
We will also need to work with \emph{long-range divided differences} $\partial_{ij}f=\frac{f-\tau_{ij}f}{x_i-x_j}$.
Unlike the $\partial_i$, the long-range $\partial_{ij}$ need not preserve Schubert-positivity.

The natural perfect degree map pairing $H_\bullet(\fl{n})\otimes H^\bullet(\fl{n})\to \mathbb{Z}$ allows us to view homology classes as homomorphisms $H^\bullet(\fl{n})\to \mathbb{Z}$. For an algebraic cycle $X\subset \fl{n}$, the operation associated to $[X]$ is the degree map $\int_X$, which is zero on $H^a(\fl{n})$ for all $a\ne \dim X$, and pairing with the fundamental class when $a=\dim X$. As shown in \cite{BGG73} this operation for $[X^w]$ is $\ct \partial_w$, and so by the duality in \eqref{eqn:ev0duality} we always have
\begin{align}\label{eq:class_of_X}
    [X]=\sum_{w\in S_n}\left(\int_X\schub{w}(x_1,\ldots,x_n)\right)[X^w].
\end{align}
We can therefore interpret $h^X_w$ as the coefficients in the expansion
$$\int_X=\sum_{w\in S_n} h^X_w \ct \partial_w,$$
where the equality is of functionals on
 $H^\bullet(\fl{n})$. This is the Kostant--Kumar approach to Schubert calculus \cite{KK86}, which allows one to work with the combinatorics of degree maps without passing to the dual combinatorics of Schubert polynomials.

An alternate representation of homology classes $[X]$ from this perspective is via \emph{degree polynomials} (\cite[Theorem 3.13]{BGG73}, \cite{PS09}), which we present in a normalization that preserves integrality. Let $z_1,\ldots,z_n$ be a formal dual basis to formal indeterminates $x_1,\ldots, x_n$. Then there is an injection $$H_\bullet(\fl{n})\hookrightarrow \mathbb{Z}[x_1,\ldots,x_n]^{\vee}\cong \mathbb{Z}[z_1,\ldots,z_n],$$
and so a homology class $[X]$ can be expressed as the polynomial $\mathcal{D}_X(z_1,\ldots,z_n)$ from \eqref{eqn:generating}. The polynomial $\mathcal{D}_w(z_1,\ldots,z_n)$ associated to $[X^w]$ is the \emph{Schubert degree polynomial}. Then we have $\langle \mathcal{D}_w(z_1,\ldots,z_n),\schub{w'}(x_1,\ldots,x_n)\rangle=\delta_{w,w'}$, and we can consequently write
\begin{align*}[X]&=\sum \langle \mathcal{D}_X(z_1,\ldots,z_n),\schub{w}(x_1,\ldots,x_n)\rangle [X^w]\text{ and }\\\mathcal{D}_X(z_1,\ldots,z_n)&=\sum \langle \mathcal{D}_X(z_1,\ldots,z_n),\schub{w}(x_1,\ldots,x_n)\rangle \mathcal{D}_w(z_1,\ldots,z_n).
\end{align*}

Finally, for $D=(d_1<d_2<\cdots<d_r)\subset [n]$ we write $\partial_D=\prod_{i=1}^r\prod_{j=1}^{i-1}\partial_{d_id_j}$, the full symmetrizer on the variables $x_{d_1},\ldots,x_{d_r}$.
If $D=\{a,b\}$ then $\partial_D=\partial_{ab}$ is the long-range divided difference, and we adopt the conventions $\partial_\emptyset=\idem$ (the identity).
Furthermore, $\partial_{[n]}=\partial_{w_\circ}$ and we have the well-known Weyl symmetrization formula $$\partial_Df(x_{d_1},\ldots,x_{d_r})=\sum_{w\in S_r}\frac{f(x_{d_{w(1)}},\ldots,x_{d_{w(r)}})}{\prod_{i<j}(x_{d_{w(i)}}-x_{d_{w(j)}})}\in \mathbb{Z}[x_{d_1},\ldots,x_{d_r}]^{S_r}.$$
We will often use the fact that if $f,g\in \mathbb{Z}[\{x_i\}_{i\in D}]$ and $g$ is symmetric with no constant term then
\begin{equation}
\label{eqn:pullsymthrough}Y_D\partial_D gf=Y_Dg\partial_Df=0,\end{equation}
where the first equality follows from the explicit description of $\partial_D$ above.

\begin{fact}
\label{fact:flagdel}
   We have
   $[\fl{n}]=\ct\partial_{w_{\circ}}$ and $\mathcal{D}_{\fl{n}}=\prod_{i<j}(z_i-z_j)$. In particular, $$\int_{\fl{n}}x_1^{a_1}\cdots x_n^{a_n}=\begin{cases}(-1)^{\ell(\sigma)}&(a_1,\ldots,a_n)=(n-\sigma(1),\ldots,n-\sigma(n))\text{ for some }\sigma\in S_n\\0&\text{otherwise.}\end{cases}$$
\end{fact}
\begin{proof}
    The first equality follows as $\fl{n}=X^{w_{\circ}}$.
    As $\partial_{w_{\circ}}w=(-1)^{\ell(w)}\partial_{w_{\circ}}$, it suffices to show for $a_1\ge \cdots \ge a_n$ that $\ct \partial_{w_{\circ}}x_{1}^{a_1}\cdots x_n^{a_n}=1$ if $(a_1,\ldots,a_n)=(n-1,\ldots,0)$ and $0$ otherwise.
    For the first case $\partial_{w_{\circ}} x_1^{n-1}x_2^{n-2}\cdots x_{n-1}=\frac{\prod_{i<j} (x_i-x_j)}{\prod_{i<j}(x_i-x_j)}=1$.
    The second case follows because if $\sum a_i=\ell(w_{\circ})=\binom{n}{2}$ then there exists $i$ such that $a_i=a_{i+1}$, so $\partial_{w_\circ}x_{1}^{a_1}\cdots x_n^{a_n}=\partial_{w_\circ s_i}\partial_ix_{1}^{a_1}\cdots x_n^{a_n}=0$.
\end{proof}

\subsection{Springer fiber components}
 As mentioned in the introduction, for a fixed nilpotent matrix $N$ of type $\lambda$ the irreducible components $\mathcal{B}_T$ of the Springer fiber $$\mathcal{B}_\lambda\coloneqq \{\mathcal{F}\in \fl{n}\suchthat N\mathcal{F}_i\subset \mathcal{F}_i\text{ for all $i$}\}\subset \fl{n}$$ were shown by Spaltenstein \cite{Spa76} to be naturally indexed by standard Young tableaux.
We draw Young diagrams in English convention, with the longest row on top.
For a partition $\lambda$ we let $\syt(\lambda)$ denote the standard fillings of $\lambda$ that are \emph{increasing} in the rows and columns (so that $1$ always appears in the top left corner).

\begin{defn}\label{def:springer_fiber_component}
    Fix $\lambda$ a partition of $n$, and let $N$ be a nilpotent matrix of type $\lambda$.
    For each $T\in \syt(\lambda)$ consider the subset $\mathring{\mathcal{B}}_T\subset \mathcal{B}_\lambda$ such that $N_{\mathbb{C}^n/V_{j}}$ has Jordan type given by the subdiagram $\lambda[n-j]\subset \lambda$ consisting of boxes labeled with numbers $\le n-j$. Then the \emph{Springer fiber component corresponding to $T$} is the closure $$\mathcal{B}_T\coloneqq \overline{\mathring{\mathcal{B}}_T}\subset \mathcal{B}_\lambda.$$
\end{defn}
We give a concrete example in Example~\ref{eg:MatchingfromTableaux}.
By \cite{Spa77}, the Springer fiber is equidimensional  and for all $T\in \syt(\lambda)$ we have
$$\dim \mathcal{B}_T=\dim \mathcal{B}_\lambda=\sum \binom{\lambda_i^{\top}}{2}.$$

\begin{rem}
\label{rem:differentconv}
   This follows the indexing conventions of Karp--Precup \cite{karp2025richardsontableauxcomponentsspringer} and Precup--Sabando-Alvarez \cite{sabandoalvarez2026idealsdefiningcomponentstworow}.
   Applying the complementation map $i\mapsto  n+1-i$ to the entries results in standard reverse tableaux that now have decreasing rows and columns, and then the indexing conventions also agree with Spaltenstein \cite{Spa76}, G\"uemes \cite{Gu89}, Fung \cite{Fun03}, and Stroppel--Webster \cite{SW12}.

   There is a genuinely different convention where the Jordan types of $N|_{V_j}$ create the standard filling -- this is the convention used by Steinberg \cite{St88}.
   As shown by van Leeuwen \cite{Vl00} the tableaux indexing the same component under the two conventions are related by the Sch\"utzenberger evacuation involution.
\end{rem}

\subsection{Orbital varieties and Joseph polynomials}
We recall now the correspondence between Springer fibers, and orbital varieties \cite{Jo84,Jo89,KZJ14,RTVZJ12}, following Spaltenstein \cite{Spa77}. Let $\bm{b},\bm{n}$ be the subsets of weakly (resp. strictly) upper triangular $n\times n$ matrices.
For $M$ nilpotent of Jordan type $\lambda$,
by definition we have $\mathcal{B}_\lambda=\{gB\in \GL_n/B\suchthat g^{-1}Mg\in \bm{b}\}$, and as $M$ is nilpotent this is equivalent to
$$\mathcal{B}_\lambda=\{gB\in \GL_n/B\suchthat g^{-1}Mg\in \bm{n}\}.$$
For $\lambda$ a partition, we denote by $\bm{n}_{\lambda}\subset \bm{n}$ for the closure of the subset of matrices of Jordan type $\lambda$.
Then if $\pi:\GL_n\to \GL_n/B$ is the projection map and $\phi:\pi^{-1}(B)\to \bm{n}$ is the map $g\mapsto g^{-1}Mg$, there is a correspondence

\begin{center}
\begin{tikzcd}
   &\pi^{-1}(\mathcal{B}_\lambda)=\{g\in \GL_n\suchthat g^{-1}Mg\in \bm{n}\}\ar[dl,"\pi",swap]\ar[dr,"\phi"]&\\ \mathcal{B}_\lambda&&\bm{n}_{\lambda}
\end{tikzcd}
\end{center}
Then $\bm{n}_{\lambda}$ decomposes into irreducible components $\mathcal{O}_T=\phi(\pi^{-1}\mathcal{B}_T)$, the \emph{orbital varieties} indexed by standard Young tableaux: $$\bm{n}_{\lambda}=\bigcup_{T\in \syt(\lambda)}\mathcal{O}_T.$$

Note that for $gB\in \mathcal{B}_M$, the associated flag $V_i$ is given by $\langle g({\sf  e}_1),\ldots,g({\sf e}_i)\rangle$, and the condition that $gB\in \mathring{\mathcal{B}}_T$ is therefore that $M_{\mathbb{C}^n/\langle g({\sf e}_1),\ldots,g({\sf e}_i)\rangle}$ is of Jordan type $\lambda[n-i]$.
This is equivalent to saying that $(g^{-1}Mg)_{\mathbb{C}^n/\langle {\sf e}_1,\ldots,{\sf e}_i\rangle}$ is of Jordan type $\lambda[n-i]$, so we have the following.
\begin{fact}
    The irreducible component $\mathcal{O}_T\subset \bm{n}_\lambda$ is the closure of those $N\in \bm{n}$ such that $N_{\mathbb{C}^n/\langle \sf{e}_1,\ldots,\sf{e}_i\rangle}$ is of Jordan type $\lambda[n-i]$ for all $i$.
\end{fact}

The variety $\bm{n}_\lambda$ is equidimensional and for all $T\in \syt(\lambda)$ we have \cite[Section 3.1]{Jo84}
$$\dim \mathcal{O}_T=\dim \bm{n}_{\lambda}=\binom{n}{2}-\sum \binom{\lambda_i^{\top}}{2}.$$

There is an action of the diagonal matrices $D$ on $\bm{n}_\lambda$ given by conjugation. Denoting $\sum m_iz_i$ for the character $\operatorname{Diag}(t_1,\ldots,t_n)\mapsto t_1^{m_1}\cdots t_n^{m_n}$, $\bm{n}$ becomes a $D$-representation of dimension $\binom{n}{2}$ with weights $\{z_i-z_j\suchthat 1\le i< j \le n\}$.
\begin{defn}
    The \emph{Joseph polynomial} \cite{Jo84,Jo89,KZJ14,RTVZJ12} $J_T(z_1,\ldots,z_n)$ of $\mathcal{O}_T$ is the $D$-equivariant cohomology class of $\mathcal{O}_T$ in $H^\bullet_D(\bm{n})=H^\bullet_D(pt)=\mathbb{Z}[z_1,\ldots,z_n]$, i.e. the \emph{multidegree} \cite{MiSt05}.
\end{defn}

\section{Specht polynomials and Levi-Richardsons in Springer's representation}
\label{sec:SpringerRepRichardson}

Given a column-strict filling $T$ of the Young diagram of $\lambda$ with distinct entries from $[n]$, its \emph{Specht polynomial} is
\[
  f_T \;=\;  \prod_{a\text{ above $b$ in the same column of }T}(z_a-z_b),
\]
the product of the Vandermonde determinants of the column entries. For example,
$$T=\begin{ytableau}
    2 & 1 & 4\\
    6 & 3 & 5\\
    7&\none&\none\end{ytableau} \quad\text{has}\quad f_T=(z_2-z_6)(z_2-z_7)(z_6-z_7)\cdot (z_1-z_3)\cdot (z_4-z_5).$$
The symmetric group $S_n$ acts on the span of the $f_T$ by permuting the variables; this span is the \emph{Specht module} $V_\lambda$.
We refer the reader to \cite{Sag01} for more on this classical construction of irreducible representations of $S_n$ \cite{Sp35}.

Let $S_{\lambda^{\top}}=\prod S_{\lambda_{i}^{\top}}\subset S_n$ be the parabolic subgroup associated to  $\lambda^{\top}$.
By \cite[Proposition 2.5]{W13}, a Richardson variety is an orbit of the Levi subgroup $\GL_{\lambda^{\top}}=\prod \GL_{\lambda_{i}^{\top}}$ if and only if it is of the form $\mathcal{R}_{u,v}$ with $u,v$ the minimal and maximal length representatives of a coset $\mathcal{C}\in S_{\lambda^{\top}}\setminus S_n$; we write $\mathcal{R}_{\mathcal{C}}\coloneqq \mathcal{R}_{u,v}$ for it.
Because these cosets partition $S_n$, these $\lambda^{\top}$-Levi Richardsons are equivalently characterized as the $\GL_{\lambda^{\top}}$-orbits through permutation matrices, as mentioned in the introduction.

Such cosets are naturally in bijection with ordered set partitions $$\mathcal{P}_1\sqcup \mathcal{P}_2\sqcup \cdots \sqcup \mathcal{P}_{\lambda_1}=[n]$$
where $|\mathcal{P}_i|=\lambda_{i}^{\top}$, by applying the inverse $w^{-1}$ of a coset representative to the elements of the standard partition $\{1,\ldots,\lambda_1^{\top}\},\{\lambda_1^{\top}+1,\ldots,\lambda_1^{\top}+\lambda_2^{\top}\},\ldots $ of $[n]$.
We obtain a column-strict filling $T$ of $\lambda$ by filling the $i$'th column of $\lambda$ with the elements of $\mathcal{P}_i$ listed in increasing order.

\begin{thm}
\label{thm:RCcohomology}
    If $T$ is the column-strict filling associated to $\mathcal{C}\in S_{\lambda^{\top}}\setminus S_n$ as above, then $$[\mathcal{R}_{\mathcal{C}}]=\ct \prod \partial_{\mathcal{P}_i}\text{ and }\mathcal{D}_{\mathcal{R}_{\mathcal{C}}}(z_1,\ldots,z_n)=f_T(z_1,\ldots,z_n)\in H_\bullet(\fl{n}).$$
\end{thm}
\begin{proof}
The Richardson variety $\mathcal{R}_{\mathcal{C}}$ is the image
$$\prod \GL_{\lambda_{i}^{\top}}\to \prod \GL_{\lambda_{i}^{\top}}/B_{\lambda_{i}^{\top}}=\prod \fl{\lambda_{i}^{\top}}\hookrightarrow \fl{n}$$
where the last map is the pattern map embedding of Bergeron--Sottile \cite{BS98,BS99} associated to the ordered partition above.
   The pushforward along the pattern map in homology is adjoint to the alphabet coproduct in cohomology, so by the push-pull formula we have
     \begin{align*}\int_{\mathcal{R}_{\mathcal{C}}}g(x_1,\ldots,x_n)=\int_{\prod_i \fl{\lambda_{i}^{\top}}}g(\{x_j\}_{j\in \mathcal{P}_i})=\Big(\prod_i Y_{\mathcal{P}_i}\partial_{\mathcal{P}_i}\Big)g=\ct \prod_i \partial_{\mathcal{P}_i}g,\end{align*}
    where the middle expression denotes the composite of the operators $Y_{\mathcal{P}_i}\partial_{\mathcal{P}_i}$ on disjoint variable sets.
    We conclude by Fact~\ref{fact:flagdel} that $\mathcal{D}_{\mathcal{R}_\mathcal{C}}=\prod_i \mathcal{D}_{\fl{\lambda_i^{\top}}}(\{z_j\}_{j\in\mathcal{P}_i})=f_T$.
\end{proof}
\begin{cor}
In $H_\bullet(\fl{n})$ the $\lambda^{\top}$-Levi Richardson classes span a submodule of $H_\bullet(\fl{n})$ isomorphic to the Specht module $V_{\lambda}$:
    $$V_\lambda\cong \mathbb{Z}\{[\mathcal{R}_{\mathcal{C}}]\suchthat \mathcal{C}\in S_{\lambda^{\top}}\setminus S_n\}\subset H_\bullet(\fl{n}).$$
\end{cor}
We will see shortly that the Springer representation is exactly this submodule of $H_\bullet(\fl{n})$.
\subsection{The geometry of Levi-Richardsons}
\label{sec:degeneration}
Martha Precup \cite{MarthaPriv} communicated to us the following remarkable degeneration. We present this using companion matrices on the suggestion and guidance of Ruizhen Liu.
Recall that in the introduction we have defined for any $n\times n$ matrix $M$ the subvariety $\mathcal{B}_M\subset \fl{n}$ consisting of flags preserved by $M$.
If we have a $1$-parameter family of matrices $M(t)$ we have a (not necessarily flat) degeneration $\mathcal{B}_{M(t)}\to \mathcal{B}_{M(0)}$.
Now fix distinct $a_1,\ldots,a_{\lambda_1}\in \mathbb{C}$ and take the block matrix $M(t)=M_1(t)\oplus M_2(t)\oplus \cdots \oplus M_{\lambda_1^{\top}}(t)$ with blocks the companion matrices (with $e_i$ the $i$'th elementary symmetric polynomial) $$M_i(t)\coloneqq \begin{bmatrix}0&0&0&\cdots &0&-e_{\lambda_i}(-ta_1,\ldots,-ta_{\lambda_i})\\1&0&0&\cdots&0&-e_{\lambda_i-1}(-ta_1,\ldots,-ta_{\lambda_i})\\
0&1&0&\cdots &0&-e_{\lambda_i-2}(-ta_1,\ldots,-ta_{\lambda_i})\\
\vdots&\vdots&\vdots&\cdots &\vdots &\vdots\\
0&0&0&\cdots & 1&-e_{1}(-ta_1,\ldots,-ta_{\lambda_i})\end{bmatrix}.$$
By definition this is the companion matrix whose characteristic polynomial is $(x-ta_1)\cdots(x-ta_{\lambda_i})$.
Then we have the observations
\begin{enumerate}
    \item For $t\ne 0$ the matrix $M_i(t)$ is diagonalizable with eigenvalues $ta_1,\ldots,ta_{\lambda_i}$.
    Therefore $M(t)$ is diagonalizable  with $ta_i$ of multiplicity $\lambda_{i}^{\top}$.
    \item $M(0)$ is a nilpotent matrix of Jordan type $\lambda$.
\end{enumerate}
For the $n\times n$ diagonal matrix $D_{\lambda^{\top}}=\operatorname{Diag}(a_1,\ldots,a_1,a_2,\ldots,a_2,\ldots)$ where $a_i$ appears $\lambda_i^{\top}$ times for $1\le i \le \lambda_1$, if $g(t)\in \GL_n$ is the matrix such that $M(t)=g(t)tD_{\lambda^{\top}}g(t)^{-1}$ then  $$\mathcal{B}_{M(t)}=g(t)\mathcal{B}_{tD_{\lambda^{\top}}}=g(t)\mathcal{B}_{D_{\lambda^{\top}}}=\bigcup \{g(t)\mathcal{R}_{\mathcal{C}}\suchthat \mathcal{C}\in S_{\lambda^{\top}}\setminus S_n\}.$$ For $t\ne 0$ this is an isotrivial family degenerating to $\mathcal{B}_\lambda$. Because the special fiber and the general fiber of $\mathcal{B}_{M(t)}$ both have dimension $\sum \binom{\lambda_{i}^{\top}}{2}$, this implies the following.

\begin{cor}
For $\mathcal{C}\in S_{\lambda^{\top}}\setminus S_n$ we have
\begin{equation}
\label{eqn:RCtoBT}[\mathcal{R}_{\mathcal{C}}]=\sum_{T\in \syt(\lambda)} g^{\mathcal{C}}_T [\mathcal{B}_T]\text{ with }g^{\mathcal{C}}_T\ge 0.
\end{equation}
\end{cor}

One can furthermore show that over $\mathbb{Q}$ the span of  $\{[\mathcal{R}_{\mathcal{C}}]\suchthat \mathcal{C}\in S_{\lambda^{\top}}\setminus S_n\}$ is the Springer representation \cite{MarthaPriv}. As we will see next, by combinatorial means we can show that the $\mathbb{Z}$-span of the $[\mathcal{R}_{\mathcal{C}}]$ is the Springer representation.

\subsection{The combinatorics of Levi-Richardsons}
Recall the quotient ring presentation of $H^{\bullet}(\mathcal{B}_{\lambda})$ due to Tanisaki \cite{Tan82} following De Concini--Procesi \cite{DCP81}.
Given $S\subset [n]$ let $e_d(\{x_j\}_{j\in S})$ denote the sum of degree $d$ squarefree monomials in the variables $\{x_j\}_{j\in S}$. Now define
\[
    \mathcal{I}_{\lambda}=\langle e_d(\{x_j\}_{j\in S}) \suchthat S\subset [n], d>|S|-\lambda_n^{\top}-\cdots - \lambda_{n-|S|+1}^{\top} \rangle,
\]
where we set $\lambda_j^{\top}=0$ for $j>\lambda_1$.
Then we have
\[
    H^\bullet(\mathcal{B}_\lambda)=\mathbb{Z}[x_1,\ldots,x_n]/\mathcal{I}_\lambda.
\]
Although Tanisaki \cite{Tan82} and De Concini--Procesi \cite{DCP81} state this presentation over a field of characteristic zero, it holds over $\mathbb{Z}$: the Springer fiber $\mathcal{B}_\lambda$ admits an affine paving (by Spaltenstein \cite{Spa76,Spa77}), so $H^\bullet(\mathcal{B}_\lambda)$ is a free $\mathbb{Z}$-module and the Tanisaki relations hold integrally -- see \cite[Theorem 4.1]{AH16} and the subsequent remark in \textit{ibid}.
\begin{thm}
\label{thm:Zupgrade}
Both the Springer representation and the span of the Levi-Richardsons associated to $\lambda^{\top}$ are equal $S_n$ sub-representations of  $H_\bullet(\fl{n})$, i.e.
$$V_\lambda\cong \bigoplus_{T\in \syt(\lambda)}\mathbb{Z}[\mathcal{B}_T]= \mathbb{Z}\{[\mathcal{R}_{\mathcal{C}}]\suchthat \mathcal{C}\in S_{\lambda^{\top}}\setminus S_n\}\subset H_\bullet(\fl{n}).$$
\end{thm}
\begin{proof}
Both sides are isomorphic to the Specht module $V_\lambda$, and the Specht polynomials form a generating set for the Specht module $V_\lambda$ over $\mathbb{F}_p$ for all $p$, so by Theorem~\ref{thm:RCcohomology} the span of the $[\mathcal{R}_{\mathcal{C}}]$ is a saturated subset of $H_\bullet(\fl{n})$ (i.e. $kv$ lies in the span implies $v$ lies in the span). It therefore suffices to show each $[\mathcal{R}_{\mathcal{C}}]$ lies in the span of the Springer components. The Springer fiber representation in homology comes from the top dimension of the inclusion
$$H_{\bullet}(\mathcal{B}_\lambda)\hookrightarrow H_\bullet(\fl{n}),$$
which is dual to the surjection
$$\mathbb{Z}[x_1,\ldots,x_n]/\langle f-\ct f\suchthat f\text{ is symmetric}\rangle=H^\bullet(\fl{n})\to H^\bullet(\mathcal{B}_\lambda)=\mathbb{Z}[x_1,\ldots,x_n]/\mathcal{I}_\lambda.$$
Therefore a homology class of $H_\bullet(\fl{n})$ lies in $H_\bullet(\mathcal{B}_\lambda)$ if and only if it annihilates $\mathcal{I}_\lambda$, and we must show that $[\mathcal{R}_{\mathcal{C}}]$ has this property.
That $f_T$ has this property when considered as a functional on the symmetric coinvariants also follows from \cite[Lemma 11]{RYZ20}, which in our language describes a full spanning set for $H_\bullet(\mathcal{B}_\lambda)\subset \mathbb{Z}[z_1,\ldots,z_n]$. We include a short proof for $[\mathcal{R}_{\mathcal{C}}]$ itself for completeness.

By Theorem~\ref{thm:RCcohomology} we know that $[\mathcal{R}_{\mathcal{C}}]=\prod Y_{\mathcal{P}_i}\partial_{\mathcal{P}_i}$. Expanding a Tanisaki generator by the alphabet coproduct associated to the partition $\mathcal{P}$, we obtain
$$e_d(\{x_i\}_{i\in S})=\sum_{\substack{a_1+\cdots+a_{\lambda_1}=d\\a_b\le |S\cap \mathcal{P}_b|}}\prod_{b=1}^{\lambda_1} e_{a_b}(\{x_i\}_{i\in S\cap \mathcal{P}_b}).$$
By \eqref{eqn:pullsymthrough} it suffices to show that the bounds on $d$ ensure that each term in the coproduct contains at least one factor of the form $e_{a_b}(\{x_i\}_{i\in S\cap \mathcal{P}_b})$ with $a_b\ge 1$ and $S\cap \mathcal{P}_b=\mathcal{P}_b$. Indeed, suppose that this never happens. If $|S|\le n-\lambda_1$, then $d>|S|$ and $e_d(\{x_j\}_{j\in S})=0$, so assume we can write $|S|=n-\lambda_1+k$ with $k\ge 0$.
Then by the pigeonhole principle there are $M\ge k$ of the $\lambda_1$ parts $\mathcal{P}_b$ completely contained in $S$ (so $a_b=0$), and the remaining $\lambda_1-M$ parts have $a_b\le |S\cap \mathcal{P}_b|\le |\mathcal{P}_b|-1$.
Therefore $$d=\sum a_i\le n-(\lambda_1-M)-\sum_{j=\lambda_1-M+1}^{\lambda_1}\lambda_j^{\top}\le n-(\lambda_1-k)-\sum_{j=\lambda_1-k+1}^{\lambda_1}\lambda_j^{\top}=|S|-\sum_{j=n-|S|+1}^n\lambda_j^{\top},$$
violating the degree bound between $d$ and $|S|$.
\end{proof}
\begin{cor}
Under the degree polynomial embedding
    $H_\bullet(\fl{n})\hookrightarrow \mathbb{Z}[z_1,\ldots,z_n]$, the Springer representation is the Specht polynomial representation spanned by the Specht polynomials $f_T$ for $T$ a column-strict filling of $\lambda$, which in turn are the degree polynomials $\mathcal{D}_{\mathcal{R}_{\mathcal{C}}}(z_1,\ldots,z_n)$ of $\lambda^{\top}$-Levi Richardsons $\mathcal{R}_{\mathcal{C}}$ for $\mathcal{C}\in S_{\lambda^{\top}}\setminus S_n$.
\end{cor}

\section{Joseph polynomials and the Springer basis}
\label{sec:Joseph}

In this section, we show that under the embedding $H_\bullet(\fl{n})\subset \mathbb{Z}[z_1,\ldots,z_n]$ we have chosen, that the Joseph polynomials $J_T(z_1,\ldots,z_n)$ associated to orbital varieties $\{\mathcal{O}_T\suchthat T\in \syt(\lambda)\}$ agree with the Springer basis $[\mathcal{B}_T]$ of homology classes of Springer fiber components.
\begin{thm}
\label{thm:JTBT}
For $T\in \syt(\lambda)$ we have $\mathcal{D}_{\mathcal{B}_T}(z_1,\ldots,z_n)=J_T(z_1,\ldots,z_n)$.
\end{thm}
\begin{cor}
\label{cor:SpechtintoJoseph}
    For $T$ a column-strict filling of $\lambda$, there is a nonnegative expansion
    $$f_T(z_1,\ldots,z_n)=\sum_{T'\in \syt(\lambda)} g^T_{T'}J_{T'}(z_1,\ldots,z_n)\text{ with }g^T_{T'}\ge 0$$
    of Specht polynomial into Joseph polynomials.
\end{cor}
\begin{proof}
    By Theorem~\ref{thm:JTBT} and Theorem~\ref{thm:RCcohomology}, this is a direct reformulation of \eqref{eqn:RCtoBT} .
\end{proof}
Joseph \cite[\S 2]{Jo89} proves a similar statement to Theorem~\ref{thm:JTBT}, identifying $\mathcal{D}_{\mathcal{B}_T}$ and $J_T$ up to a nonzero scalar; the content of Theorem~\ref{thm:JTBT} is that with our normalizations this scalar is $1$.

\begin{proof}[Proof of Theorem~\ref{thm:JTBT}]
    Hotta \cite{Ho84} showed that there is an isomorphism of $S_n$-representations
    $$\bigoplus_{T\in \syt(\lambda)}\mathbb{Z}[\mathcal{B}_T]\cong \bigoplus_{T\in \syt(\lambda)}\mathbb{Z}J_T(z_1,\ldots,z_n)$$
    with $[\mathcal{B}_T]\mapsto J_T(z_1,\ldots,z_n)$. Therefore because the Specht module $V_\lambda$ is irreducible after tensoring with $\mathbb{Q}$, it suffices to show for a single $T$ that $\mathcal{D}_{\mathcal{B}_T}(z_1,\ldots,z_n)=J_T(z_1,\ldots,z_n)$. We do this for the filling where the $i$'th column contains the interval $\mathcal{P}_i=\{\lambda_1^{\top}+\cdots+\lambda_{i-1}^{\top}+1,\ldots,\lambda_1^{\top}+\cdots+\lambda_{i-1}^{\top}+\lambda_{i}^{\top}\}$. We choose the nilpotent matrix $M$ which takes $\sf{e}_{n+1-b}\mapsto \sf{e}_{n+1-a}$ if the box containing $a$ is directly to the right of the box containing $b$ (and to $\sf{e}_{n+1-b}\mapsto 0$ if $b$ is in the last box of its row). Then one can directly check $(\GL_{\lambda^{\top}_{\lambda_1}}\times\GL_{\lambda^{\top}_{\lambda_1-1}}\times \cdots \times \GL_{\lambda^{\top}_1})\cdot \idem\subset \mathcal{B}_T$. But the $\GL$ product is equal to $\mathcal{R}_{\mathcal{C}}$ for $\mathcal{C}\in S_{\lambda^{\top}}\setminus S_n$ the coset associated to the column-strict filling $T'$ obtained from $T$ by reversing each column and replacing each label $i$ with $n+1-i$. This Richardson variety is closed and irreducible and of dimension $\sum \binom{\lambda_i^{\top}}{2}$, so is therefore equal to $\mathcal{B}_T$. By Theorem~\ref{thm:RCcohomology} we therefore have
$\mathcal{D}_{\mathcal{B}_T}=\mathcal{D}_{\mathcal{R}_{\mathcal{C}}}=f_{T'}\in H_\bullet(\fl{n}).$

We now compare this to the multidegree of the orbital variety.
For this particular $T$ it is straightforward to see that any generic matrix in $\bm{n}$ with $0$ entries in positions $(n+1-a,n+1-b)$ whenever $a,b$ lie in the same part $\mathcal{P}_i$ satisfies the Jordan condition to lie in $\mathcal{O}_T$.
The closure is a coordinate subspace of dimension $\binom{n}{2}-\sum \binom{\lambda_i^{\top}}{2}$, so must equal  $\mathcal{O}_T$. Because $\mathcal{O}_T\subset \bm{n}$ is a coordinate sub-$D$-representation, the multidegree is the product of the characters of the coordinates set to zero \cite{MiSt05} and as this is equal to $f_{T'}$ we conclude.
\end{proof}
\begin{eg}
    For $\lambda=(4,2,1)$ we have
    $$T=\begin{ytableau}
        1&4&6&7\\
        2&5\\
        3
    \end{ytableau}\quad T'=\begin{ytableau}
        5&3&2&1\\
        6&4\\
        7
    \end{ytableau}\quad \text{and }M:\begin{cases}{\sf e}_7\mapsto {\sf e}_4\mapsto {\sf e}_2\mapsto {\sf e}_1\mapsto 0&\\
    {\sf e}_6\mapsto {\sf e}_3\mapsto 0\\
    {\sf e}_5\mapsto 0\end{cases}$$
    $f_{T'}=(z_5-z_6)(z_5-z_7)(z_6-z_7)(z_3-z_4)$, and the corresponding component $\mathcal{B}_T\subset \fl{7}=\gl_7/B_7$ and orbital variety $\mathcal{O}_T\subset \bm{n}$ are respectively:
    $$\begin{bmatrix}
    \ast&0&0&0&0&0&0\\
    0&\ast&0&0&0&0&0\\
    0&0&\ast&\ast&0&0&0\\
    0&0&\ast&\ast&0&0&0\\
    0&0&0&0&\ast&\ast&\ast\\
    0&0&0&0&\ast&\ast&\ast\\
    0&0&0&0&\ast&\ast&\ast\end{bmatrix}B_7\text{ and }\begin{bmatrix}0&\ast&\ast&\ast&\ast&\ast&\ast\\
    0&0&\ast&\ast&\ast&\ast&\ast\\
    0&0&0&\boxed{0}&\ast&\ast&\ast\\
    0&0&0&0&\ast&\ast&\ast\\
    0&0&0&0&0&\boxed{0}&\boxed{0}\\
    0&0&0&0&0&0&\boxed{0}\\
    0&0&0&0&0&0&0\end{bmatrix}\subset \begin{bmatrix}0&\ast&\ast&\ast&\ast&\ast&\ast\\
    0&0&\ast&\ast&\ast&\ast&\ast\\
    0&0&0&\ast&\ast&\ast&\ast\\
    0&0&0&0&\ast&\ast&\ast\\
    0&0&0&0&0&\ast&\ast\\
    0&0&0&0&0&0&\ast\\
    0&0&0&0&0&0&0\end{bmatrix}.$$
    The blocks of $\ast$ in $\mathcal{B}_T\subset \gl_{7}/B_7$ have to be invertible, but there is no constraint on $\ast$ in $\mathcal{O}_T\subset \bm{n}$.
\end{eg}

\section{Two-row Springer fiber components}

We now fix $\lambda=(n-k,k)$, and prove combinatorially nonnegative expansions for the $\lambda^{\top}$-Levi Richardsons $\mathcal{R}_{\mathcal{C}}$ with $\mathcal{C}\in S_{\lambda^{\top}}\setminus S_n$ into Springer fiber components $\mathcal{B}_T\subset \mathcal{B}_\lambda$. This resolves Question~\ref{question:PositiveExpansions} (\ref{it:RtoB}) for two-row $\lambda$.

\subsection{Preliminaries on two-row Springer fibers}
We will fix $n,k$ with $k\le n-k$ and $\lambda=(n-k,k)$. Then tableaux in $\syt(\lambda)$ are naturally in bijection with \emph{packed noncrossing matchings} $\mathcal{M}=\{(i_1<j_1),(i_2<j_2),\ldots,(i_k<j_k)\}$ of $[n]$, meaning
\begin{enumerate}
\item (Matching) $i_1,\ldots,i_k,j_1,\ldots,j_k\in [n]$ are all distinct.
    \item (Noncrossing) We never have $i_a<i_b<j_a<j_b$
    \item (Packed) If $p\not\in\{i_1,\ldots,i_k,j_1,\ldots,j_k\}$ then there is no $1\le a \le k$ such that $i_a<p<j_a$.
\end{enumerate}
We call each pair $(i_a<j_a)\in\mathcal{M}$ an \emph{arc}.
The bijection associates to a packed noncrossing matching the unique element of $\syt(\lambda)$ with bottom row $j_1<j_2<\cdots<j_k$.
In the other direction, given $T\in \syt(\lambda)$ we form a string of parentheses by making the $\ell$th parenthesis open or closed according as $\ell$ lies in the top or bottom row of $T$; the index pairs $(i,j)$ of matched parentheses then give the packed noncrossing matching.

\begin{rem}This identification appears in Precup--Sabando-Alvarez \cite{sabandoalvarez2026idealsdefiningcomponentstworow}, as well as in both Fung \cite{Fun03} and Stroppel--Webster \cite{SW12}. Despite using superficially different conventions for standard Young tableaux (Remark ~\ref{rem:differentconv}), the noncrossing matchings that index irreducible components of $\mathcal{B}_\lambda$ are exactly equal in both cases.
\end{rem}

For $\mathcal{M}$ a noncrossing matching and two arcs $\alpha=(i_a<j_a)$ and $\beta=(i_b<j_b)$ we say that $\alpha$ \emph{nests} $\beta $ (or $\beta$ is \emph{nested} under $\alpha$) if $i_a<i_b<j_b<j_a$, and write $\beta \prec \alpha$.
The relation $\prec$ is a partial ordering and we say that an ordering of $\mathcal{M}$ is a \emph{big-to-small ordering} if $\beta\prec \alpha$ means $\beta$ appears later than $\alpha$.
One such ordering lists the arcs in weakly decreasing order of length $j-i$.
When we write a product over all arcs $(i,j)\in \mathcal{M}$ it is assumed that the product is taken with respect to a big-to-small ordering.

Figure~\ref{fig:matching_and_syt} depicts both  a packed noncrossing matching and its corresponding SYT of shape $(3,3)$. A big-to-small ordering can be taken to be  $(1,6)\succ (2,3)\succ (4,5)$.

\begin{figure}[!htbp]
  \ytableausetup{baseline}
\begin{tikzpicture}[scale = 0.75, baseline = 10pt]
\foreach \x in {1, ..., 6}{\draw[fill] (\x - 1, 0) node[inner sep = 2pt] (\x) {$\scriptstyle \x$};}
\foreach \i\j in {1/6, 2/3, 4/5}{\draw[thick] (\i) to[out = 45, in = 135] (\j);}
\end{tikzpicture} \qquad\qquad
\begin{ytableau}
    1 & 2 & 4\\
    3 & 5 & 6
\end{ytableau}
\caption{A packed noncrossing matching on $\{1,\dots,6\}$ and its associated SYT. Note this is not the matching $\{(1,3),(2,5),(4,6)\}$ obtained by reading the columns.}
  \label{fig:matching_and_syt}
\end{figure}

We will carry the matching of Figure~\ref{fig:matching_and_syt} as a running example.
We revisit Definition~\ref{def:springer_fiber_component} to illustrate the various notions.
\begin{eg}
\label{eg:MatchingfromTableaux}
 Here $n=6$ and  $\lambda=(3,3)$, so we fix a nilpotent
$N$ of Jordan type $(3,3)$, i.e.\ $N^3=0$ with $\rank N=4$ and
$\rank N^2=2$. The associated tableau  records the
Jordan type of $N$ on the successive quotients $\mathbb{C}^6/V_{j}$: the subdiagrams
$\lambda[6-j]$ of boxes labeled $\le 6-j$ have shapes
$(3,3),(3,2),(3,1),(2,1),(2),(1),\varnothing$ for $j=0,1,\dots,6$, as can be verified from Figure~\ref{fig:matching_and_syt}.
\end{eg}

Given a packed noncrossing matching $\mathcal{M}$, write $\mathcal{B}_\mathcal{M}$ for the irreducible component $\mathcal{B}_T$ of the Springer fiber indexed by the tableau $T$ corresponding to $\mathcal{M}$; these are the cycles to which we apply the degree-map calculus of Section~\ref{subsec:divdiff}.
As shown by Fung \cite[Proposition 5.1]{Fun03}, each $\mathcal{B}_\mathcal{M}$ is an iterated $\mathbb{P}^1$-bundle of dimension $k=|\mathcal{M}|$.

Fung \cite{Fun03} also first proved the following characterization. In what follows we note that $j-i$ is odd for each arc $(i,j)\in \mathcal{M}$ as $\{i+1,\ldots,j-1\}$ are perfectly matched to each other.

\begin{prop}[{\cite[Proposition 7]{SW12}}]
For $a\in[n]\setminus \bigcup \mathcal{M}$ an unmatched point, let $c(a)$ be the number of arcs $(i,j)\in\mathcal{M}$ with $j<a$, and let $r(a)$ be the number of unmatched $b\in [n]\setminus \bigcup \mathcal{M}$ with $b\le a$.
Then a flag $V_\bullet$ lies in the two-row Springer fiber component $\mathcal{B}_\mathcal{M}$ if and only if
\begin{enumerate}
    \item $N^{\frac{j-i+1}{2}}(V_j)=V_{i-1}$ for every arc $(i,j)\in\mathcal{M}$, and
    \item $V_a=N^{-c(a)}\bigl(\operatorname{im} N^{\,n-k-r(a)}\bigr)$ for every $a\notin\bigcup\mathcal{M}$.
\end{enumerate}
\label{prop:fung_characterization}
\end{prop}

\begin{eg}\label{eg:springer}
Continuing with the matching $\mathcal{M}$ of
Figure~\ref{fig:matching_and_syt}, let $N$ be the nilpotent $6\times 6$ matrix of Jordan type $(3,3)$ acting on the standard basis vectors by
\[
{\sf e}_3\mapsto {\sf e}_2\mapsto {\sf e}_1\mapsto 0,\qquad {\sf e}_6\mapsto {\sf e}_5\mapsto {\sf e}_4\mapsto 0.
\]
Then $\mathcal{B}_\mathcal{M}$ is an iterated $\mathbb{P}^1$-bundle of dimension $|\mathcal{M}|=3$. Since every point of $[6]$ is matched, the second family of conditions in
Proposition~\ref{prop:fung_characterization} is vacuous.
The component
$\mathcal{B}_{\mathcal{M}}$ consists of the flags $V_\bullet$ in $\mathbb{C}^6$ with
\[
N^{3}(V_6)=V_0,\qquad N(V_3)=V_1,\qquad N(V_5)=V_3,
\]
the first holding automatically as $N^3=0$. One such flag is
\[
\langle {\sf e}_1\rangle\subset
\langle {\sf e}_1,{\sf e}_2\rangle\subset
\langle {\sf e}_1,{\sf e}_2,{\sf e}_4\rangle\subset
\langle {\sf e}_1,{\sf e}_2,{\sf e}_4,{\sf e}_3\rangle\subset
\langle {\sf e}_1,{\sf e}_2,{\sf e}_4,{\sf e}_3,{\sf e}_5\rangle,
\]
for which $N(V_3)=\langle {\sf e}_1\rangle=V_1$ and
$N(V_5)=\langle {\sf e}_1,{\sf e}_2,{\sf e}_4\rangle=V_3$.
\end{eg}
\subsection{Degree polynomials of two-row Springer fiber components and the web basis}
For a matching $\mathcal{M}$, we denote the \emph{Specht polynomial associated to the matching}
\begin{equation}
\label{eqn:fM}f_{\mathcal{M}}=\prod_{(i<j)\in \mathcal{M}}(z_i-z_j).\end{equation}
If $|\mathcal{M}|=k$ this is the Specht polynomial associated to any column-strict filling of $\lambda=(n-k,k)$ where the pairs $(i,j)\in \mathcal{M}$ occupy the columns of size $2$.
\begin{thm}
\label{thm:zizjhomology}
    For $\mathcal{M}$ a packed noncrossing matching, the degree polynomial of $\mathcal{B}_{\mathcal{M}}$ is given by  $$\mathcal{D}_{\mathcal{B}_\mathcal{M}}=f_\mathcal{M}(z_1,\ldots,z_n)\in H_\bullet(\fl{n}).$$
   In particular $h^{\mathcal{B}_\mathcal{M}}_w=\langle f_{\mathcal{M}},\schub{w}(x_1,\ldots,x_n)\rangle$ and
$[\mathcal{B}_\mathcal{M}]=\sum_{w\in S_n}\langle f_\mathcal{M},\schub{w}(x_1,\ldots,x_n)\rangle[X^w].
$
\end{thm}
\begin{rem}
\label{rem:AlternateJoseph}
    We note that by Section~\ref{sec:Joseph} it suffices to identify the corresponding Joseph polynomial $J_T(z_1,\ldots,z_n)=f_{\mathcal{M}}$, which we can read off from \cite[Equation 4.12]{RTVZJ12}, but the presentation of the cohomology ring and the Poincar\'e dual to the fundamental class allows one to also extract more refined information that may be of independent interest (see the discussion at the end of Section~\ref{sec:Schubpostworow}).
\end{rem}
To show this theorem, we will need a presentation of $H^\bullet(\mathcal{B}_\mathcal{M})$ due to Stroppel--Webster \cite{SW12}.
\begin{thm}[{\cite[Theorem 9]{SW12}}]\label{thm:SWpresentation}
    Denote the inclusion  $\iota:\mathcal{B}_\lambda\hookrightarrow \fl{n}$. There is an isomorphism
    $$H^\bullet(\mathcal{B}_\mathcal{M})\cong \mathbb{Z}[y_{i_1},\ldots,y_{i_k}]/(y_{i_1}^2,\ldots,y_{i_k}^2),$$
    with $y_{i_1}\cdots y_{i_k}$ Poincar\'e dual to the fundamental class of $\mathcal{B}_\mathcal{M}$, and the
    pullback map $$\iota^*:H^\bullet(\fl{n})\to H^\bullet(\mathcal{B}_\mathcal{M})\text{ takes }\iota^*(x_b)=\begin{cases}y_{i_a}&b=i_a\text{ for some $a$}\\-y_{i_a}&b=j_a\text{ for some $a$}\\0&\text{otherwise.}\end{cases}$$
\end{thm}
\begin{proof}
    The presentation is due to Stroppel--Webster \cite[Theorem 9]{SW12}. The Poincar\'e duality statement follows because as defined in \textit{ibid.} the $y_{i_a}$ are the relative hyperplane classes of each stage of the iterated $\mathbb{P}^1$-bundle description of $\mathcal{B}_\mathcal{M}$.

    We take the opportunity here to slightly correct and simplify the part of the proof of \cite[Theorem 9]{SW12} where it is deduced that $\iota^*x_{i}+\iota^*x_{j}=0$ for any arc $(i,j)\in \mathcal{M}$; the second displayed short exact sequence therein does not hold in general.
    Write $\delta=(j-i+1)/2$ and observe that $\delta\leq k\leq n-k$.
    As in \textit{ibid.} the component $\mathcal{B}_\mathcal{M}$ has the property that for any arc $(i,j)\in \mathcal{M}$ there is a short exact sequence $$0\to \mathcal{K}\to \mathcal{F}_j|_{\mathcal{B}_\mathcal{M}}\xrightarrow{N^{\delta}} \mathcal{F}_{i-1}|_{\mathcal{B}_\mathcal{M}}\to 0$$
    for some vector bundle $\mathcal{K}$.
    This implies $\mathcal{K}$ is a rank $2\delta$ subbundle of $\ker(N^{\delta})|_{\mathcal{B}_\mathcal{M}}$, which is a trivial bundle of rank $\min(2\delta,\delta+k)=2\delta$.
    Therefore these two bundles are equal and $\mathcal{K}$ is trivial.
    Taking first Chern classes shows  $\iota^*x_i+\cdots+\iota^*x_j=0$.
    Finally, fix an arc $(i,j)$ and consider the arcs $(i',j')$ nested under it that are maximal with respect to $\prec$.
    Subtracting their relations $\iota^*x_{i'}+\cdots+\iota^*x_{j'}=0$ from that of $(i,j)$ cancels every interior term, leaving $\iota^*x_i+\iota^*x_j=0$.
\end{proof}

    \begin{eg}\label{eg:swpresentation}
For the matching $\mathcal{M}=\{(1,6),(2,3),(4,5)\}$ of
Example~\ref{eg:springer}, we have $i_1=1$, $i_2=2$, $i_3=4$, so
Theorem~\ref{thm:SWpresentation} gives
\[
H^\bullet(\mathcal{B}_{\mathcal{M}})\cong \mathbb{Z}[y_1,y_2,y_4]/(y_1^2,\,y_2^2,\,y_4^2),
\]
with $y_1y_2y_4$ Poincar\'e dual to the fundamental class, in agreement with
$\dim \mathcal{B}_{\mathcal{M}}=3$. The pullback $\iota^*\colon H^\bullet(\fl{6})\to
H^\bullet(\mathcal{B}_{\mathcal{M}})$ acts on the generators $x_1,\dots,x_6$ by
\[
x_1\mapsto y_1,\quad x_6\mapsto -y_1,\qquad
x_2\mapsto y_2,\quad x_3\mapsto -y_2,\qquad
x_4\mapsto y_4,\quad x_5\mapsto -y_4,
\]
i.e. $(x_{i_a},x_{j_a})\mapsto (y_{i_a},-y_{i_a})$, for each arc
$(i_a,j_a)\in \mathcal{M}$. We see $\iota^*x_{2}+\iota^*x_{3}=\iota^*x_{4}+\iota^*x_{5}=0$ and
$$\iota^*x_1+\iota^*x_6=\iota^*(x_1+\cdots+x_6)-(\iota^*x_2+\iota^*x_3)-(\iota^*x_4+\iota^*x_5)=0.$$
\end{eg}
\begin{proof}[Proof of Theorem~\ref{thm:zizjhomology}]
    By Theorem~\ref{thm:SWpresentation} the degree map factors as
$$\int_{\mathcal{B}_\mathcal{M}}f=[\,y_{i_1}\cdots y_{i_k}\,]\,\iota^*f,$$
where $\iota^*x_{i_a}=y_{i_a}$, $\iota^*x_{j_a}=-y_{i_a}$, $\iota^*x_b=0$ for
$b\notin\bigcup\mathcal{M}$, and $y_{i_a}^2=0$. A monomial of $f$ thus
restricts to zero unless it omits every $x_b$ with $b\notin\bigcup\mathcal{M}$ and is
linear in exactly one of $x_{i_a},x_{j_a}$ for each $a$.
The coefficient of
$y_{i_1}\cdots y_{i_k}$ then collects precisely these contributions, giving
\begin{equation}\label{eq:signsum}
\int_{\mathcal{B}_\mathcal{M}}f
=\sum_{\epsilon\in\{0,1\}^k}(-1)^{\sum_a\epsilon_a}
\big[\,x_{i_1}^{1-\epsilon_1}\cdots x_{i_k}^{1-\epsilon_k}
x_{j_1}^{\epsilon_1}\cdots x_{j_k}^{\epsilon_k}\,\big]f=\langle \prod_{a=1}^k(z_{i_a}-z_{j_a}),f\rangle
\end{equation}
where $\epsilon_a=0$ records the $x_{i_a}$-linear part and $\epsilon_a=1$ the
$x_{j_a}$-linear part.
\end{proof}

\begin{eg}
\label{eg:132645}
    We take our running example of $\mathcal{M}$ from Figure~\ref{fig:matching_and_syt} to illustrate Theorem~\ref{thm:zizjhomology}.
    We compute $h^{\mathcal{B}_\mathcal{M}}_w$ for $w=132645$. Note that $\mathfrak{S}_{132645}(x_1,\ldots,x_6)=(x_1+x_2)h_2(x_1,x_2,x_3,x_4)$ where $h_i$ denotes the $i$th complete homogeneous symmetric polynomial, and so
    \[
        h^{\mathcal{B}_\mathcal{M}}_w=\langle (z_1-z_6)(z_2-z_3)(z_4-z_5), \mathfrak{S}_{132645}(x_1,\ldots,x_6)\rangle=2-1=1,
    \]
    with the squarefree monomial $x_1x_2x_4$ contributing a $2$ and the $x_1x_3x_4$ contributing a $-1$.
\end{eg}
In the two-row case the ordered set partitions $\mathcal{C}$ indexing the Specht polynomial/Richardson generators for the Specht module $V_\lambda$ are matchings $\mathcal{M}$.
\begin{cor}
\label{cor:BmRcequation}
    For $\lambda=(n-k,k)$ a two-row partition, if the underlying matching $\mathcal{M}$ for $\mathcal{C}\in S_{\lambda^{\top}}\setminus S_n$ is noncrossing and packed, then we have an equality
$$[\mathcal{B}_{\mathcal{M}}]=[\mathcal{R}_{\mathcal{C}}]\in H_\bullet(\fl{n}).$$
\end{cor}
\begin{proof}
Theorem~\ref{thm:RCcohomology}, \eqref{eqn:fM}, and Theorem~\ref{thm:zizjhomology} imply $\mathcal{D}_{\mathcal{R}_{\mathcal{C}}}=\prod_{(i<j)\in \mathcal{M}}(z_i-z_j)=f_{\mathcal{M}}=\mathcal{D}_{\mathcal{B}_\mathcal{M}}$. Here the tableau $T$ indexing $\mathcal{B}_\mathcal{M}$ (as described earlier) and the column-strict filling $T(\mathcal{C})$ of Theorem~\ref{thm:RCcohomology} (obtained by reading the columns of the ordered set partition) are in general different tableaux; it is only their degree polynomials that agree, both being $f_\mathcal{M}$.
\end{proof}
For arbitrary matchings $\mathcal{M}$, following Rhoades \cite{Rho19} we can combinatorially expand the Specht polynomials $f_{\mathcal{M}}$ associated to matchings \eqref{eqn:fM} into the web basis of those $f_\mathcal{M}$ corresponding to packed noncrossing matchings using the Ptolemy/skein relations.
The straightening rules are
\begin{enumerate}
    \item $(z_a-z_c)(z_b-z_d)\mapsto (z_a-z_d)(z_b-z_c)+(z_a-z_b)(z_c-z_d)$ if $a<b<c<d$
    \item $z_a-z_c\mapsto (z_a-z_b)+(z_b-z_c)$ if $a<b<c$ and $b$ is unmatched.
\end{enumerate}
The monovariant $z_a\mapsto a$ shows that this process terminates.
\begin{cor}[{Question~\ref{question:PositiveExpansions} (\ref{it:RtoB}) for two-row partitions}]
\label{cor:expandRCintoSchuberts}
    For $\lambda=(n-k,k)$ a two-row partition and $\mathcal{C}\in S_{\lambda^{\top}}\setminus S_n$
    the coefficients in the expansion $$[\mathcal{R}_{\mathcal{C}}]=\sum_{\text{packed noncrossing matchings } \mathcal{M}} g^{\mathcal{M}}_{\mathcal{C}}[\mathcal{B}_{\mathcal{M}}]\in H_\bullet(\fl{n})$$
    are nonnegative, and are computed by an explicit combinatorial algorithm.
\end{cor}
\begin{rem}
\label{rem:JosephintoDegree}
For $\lambda=(n-k,k)$, Corollary~\ref{cor:expandRCintoSchuberts} shows that the coefficients $g_{T'}^T$ in the nonnegative expansion of Specht polynomials $f_T(z_1,\ldots,z_n)$ into Joseph polynomials $J_{T'}(z_1,\ldots,z_n)$ from Corollary~\ref{cor:SpechtintoJoseph} are computed via a combinatorially nonnegative algorithm.
\end{rem}

\section{Schubert positivity of two-row Springer fiber components}
\label{sec:Schubpostworow}
Throughout this section, fix a packed noncrossing matching
$\mathcal{M}=\{(i_1,j_1),\dots,(i_k,j_k)\}$ of $[n]$ with $i_a<j_a$ for $1\leq a\leq k$ associated to a $T\in \syt((n-k,k))$.
Our main result computes the combinatorially positive Schubert cycle expansion for the homology class of $\mathcal{B}_T=\mathcal{B}_\mathcal{M}$ in $\fl{n}$, confirming a recent conjectural formula of Precup and Sabando-Alvarez. This resolves Springer's question (Question~\ref{question:PositiveExpansions} (\ref{it:SpringerQuestion})) for two-row partitions, and together with Corollary~\ref{cor:expandRCintoSchuberts} resolves Question~\ref{question:RctoSchubert} for two-row partitions.

\begin{thm}[{\cite[Conjecture 2]{sabandoalvarez2026idealsdefiningcomponentstworow} and Question~\ref{question:PositiveExpansions} (\ref{it:SpringerQuestion}) for two-row partitions}]\label{thm:mainconjecture} Assuming $\mathcal{M}$ is listed according to a big-to-small ordering, the coefficient $h^{\mathcal{B}_T}_w=h^{\mathcal{B}_\mathcal{M}}_w$ in the Schubert cycle expansion  $[\mathcal{B}_\mathcal{M}]=\sum h^{\mathcal{B}_{\mathcal{M}}}_w[X^w]$ is the number of reduced words $s_{a_1}\cdots s_{a_k}$ for $w$ where $i_p\le a_p<j_p$ for all $1\leq p\leq k$. Equivalently
\begin{align*}h^{\mathcal{B}_\mathcal{M}}_w=\ct \prod_{(i<j)\in \mathcal{M}}(\partial_i+\partial_{i+1}+\cdots+\partial_{j-1})\schub{w}\quad\text{or}\quad[\mathcal{B}_{\mathcal{M}}]&=\ct \prod_{(i<j)\in \mathcal{M}}(\partial_i+\partial_{i+1}+\cdots+\partial_{j-1}).\end{align*}
\end{thm}

\begin{rem}
\label{rem:PSAdictionary}
We spell out the correspondence with \cite[Conjecture~2]{sabandoalvarez2026idealsdefiningcomponentstworow}, since the two statements are phrased differently. In the notation of \textit{ibid.}, the conjecture asserts that $[\mathcal{B}_\sigma]$ expands in the Schubert basis with coefficients counting reduced words $s_{a_1}\cdots s_{a_k}$ that are compatible with the arcs, where the arcs (``cups'') are processed in a total order in which \textit{every cup nested inside a given cup appears before it} (a small-to-big ordering), and the Schubert classes are indexed cohomologically. Our Theorem~\ref{thm:mainconjecture} uses the opposite, big-to-small ordering of the arcs and indexes cycles by $[X^w]=[\overline{BwB}]$. The two formulations are matched by the standard homology/cohomology duality $w\leftrightarrow w_\circ w$ on $S_n$, under which reversing the order of the reduced word interchanges the two cup orderings; the flagged reduced-word counts then coincide. For instance, for $\mathcal{M}=\{(1,4),(2,3)\}$ our Theorem~\ref{thm:mainconjecture} gives $[X^{2314}]+[X^{1423}]$, matching $\mathfrak{S}_{3241}+\mathfrak{S}_{4132}$ of \textit{ibid.}\ under $w\mapsto w_\circ w$. While \cite{sabandoalvarez2026idealsdefiningcomponentstworow} states the conjecture for the rectangular shape $(\ell,\ell)$ and reduces the general two-row case to it, Theorem~\ref{thm:mainconjecture} treats all two-row shapes directly.
\end{rem}

\begin{cor}[{Question~\ref{question:RctoSchubert} for two-row partitions}]
\label{cor:RCtoSchub}
    There is a combinatorially nonnegative algorithm to compute the Schubert cycle expansion coefficients of $[\mathcal{R}_{\mathcal{C}}]$ given by $[\mathcal{R}_{\mathcal{C}}]=\sum h^{\mathcal{R}_\mathcal{C}}_w [X^w]$.
\end{cor}
\begin{proof}
Corollary~\ref{cor:expandRCintoSchuberts} combinatorially expands $[\mathcal{R}_{\mathcal{C}}]$ into $[\mathcal{B}_\mathcal{M}]$ cycles and then these expand as $[\mathcal{B}_\mathcal{M}]=\sum_{w\in S_n}h^{\mathcal{B}_{\mathcal{M}}}_w[X^w]$ with coefficients the flagged reduced word counts in Theorem~\ref{thm:mainconjecture}.
\end{proof}
\begin{rem}
    For $\lambda=(n-k,k)$, Theorem~\ref{thm:mainconjecture} shows equivalently that there is a combinatorially nonnegative algorithm for the coefficients of the expansion of Joseph polynomials into Schubert degree polynomials
    $$J_T(z_1,\ldots,z_n)=\sum_{w\in S_n}h_w^T\mathcal{D}_w(z_1,\ldots,z_n)\text{ with }h_w^T\ge 0.$$
   Combined with Remark~\ref{rem:JosephintoDegree}, this gives the equivalent statement to Corollary~\ref{cor:RCtoSchub} that Specht polynomials $f_T$ have combinatorially nonnegative expansions into Schubert degree polynomials $\mathcal{D}_w$.
\end{rem}

\begin{eg}\label{eg:running}
    We compute the Schubert expansion of the component $\mathcal{B}_{\mathcal{M}}$ with $\mathcal{M}$ from
    Figure~\ref{fig:matching_and_syt}, with arcs ordered $(1,6)\succ(2,3)\succ(4,5)$.
The degree map corresponding to $\mathcal{B}_\mathcal{M}$ is given by $(\partial_1+\cdots+\partial_5)\partial_2\partial_4$.
To compute the reduced words $s_{a_1}s_{a_2}s_{a_3}$ that contribute we impose $1\leq a_1<6$, $2\leq a_2<3$, and $4\leq a_3<5$.
This yields the set $\{s_1s_2s_4, s_3s_2s_4,s_5s_2s_4\}$.
We thus obtain
\[
    [\mathcal{B}_\mathcal{M}]=[X^{231546}]+[X^{142536}]+[X^{132645}]\in H_\bullet(\fl{6}),
\]
where we have written each $w$ above in one-line notation $w(1)w(2)\cdots w(6)$. Recall that we computed the coefficient of $[X^{132645}]$ to be $1$ earlier in Example~\ref{eg:132645}.
\end{eg}

A key insight to our proof of Theorem~\ref{thm:mainconjecture} is that by Corollary~\ref{cor:BmRcequation} and Theorem~\ref{thm:RCcohomology}
we have
\begin{equation}
\label{eqn:BMispartijprod}
[\mathcal{B}_\mathcal{M}]=\ct\prod_{(i<j)\in \mathcal{M}}\partial_{ij}.
\end{equation}
Recall the operation $Y_Df$ sets the variables $x_i=0$ for $i\in D$.
\begin{prop}
\label{prop:YdelY}
If $1\le i<j\le n$ and both $\{i+1,\ldots,j-1\}\subset D\subset [n]$ and $\{i,j\}\cap D=\emptyset$, then
$$Y_{ij}\partial_{ij}Y_{D}=Y_{D\sqcup \{i,j\}}\sum_{k=i}^{j-1}\partial_k.$$
\end{prop}
\begin{proof}
    Both sides are operators on $\mathbb{Z}[x_1,\dots,x_n]$; the variables outside $\{i,i+1,\dots,j\}$ are spectators for $\partial_{ij}$ and for each $\partial_k$ with $i\le k<j$, so it suffices to check the identity on a function $f(\{x_a\}_{a\in D\sqcup \{i,j\}})$ of the remaining variables. For such $f$, write $[x_a^1]f$ for the coefficient of $x_a^1$ in $f$ after setting the other variables in $D$ to $0$. Applying the left-hand side gives $[x_i^1]f-[x_j^1]f$, and the $k$'th term on the right-hand side gives $[x_k^1]f-[x_{k+1}^1]f$, so the result follows by telescoping the sum.
\end{proof}
We are now ready to prove our main theorem, resolving \cite[Conjecture 2]{sabandoalvarez2026idealsdefiningcomponentstworow}.
\begin{proof}[Proof of Theorem~\ref{thm:mainconjecture}]
By \eqref{eqn:BMispartijprod} it suffices to prove that $\ct \prod_{(i<j)\in \mathcal{M}}\partial_{ij}=\ct \prod_{(i<j)\in \mathcal{M}}\sum_{a=i}^{j-1}\partial_a.$
Since $Y_{i_aj_a}$ commutes with $\partial_{i_bj_b}$ for $a\ne b$ (they involve disjoint variables) and $\big(\prod_{a}Y_{i_aj_a}\big)Y_{[n]\setminus\bigcup\mathcal{M}}=Y_{[n]}=\ct$, we may interleave the zero-substitutions and write
\[
\ct\prod_{(i<j)\in\mathcal{M}}\partial_{ij}=Y_{i_1j_1}\partial_{i_1j_1}Y_{i_2j_2}\partial_{i_2j_2}\cdots Y_{i_kj_k}\partial_{i_kj_k}Y_{[n]\setminus\{i_1,\dots,i_k,j_1,\dots,j_k\}}.
\]
Because our ordering on $\mathcal{M}$ is a big-to-small ordering, we know that $\{i_b+1,\ldots,j_b-1\}\subset [n]\setminus \{i_1,\ldots,i_b,j_1,\ldots,j_b\}$  for all $1\le b \le k$, so in particular by Proposition~\ref{prop:YdelY} we have $$Y_{i_bj_b}\partial_{i_bj_b}Y_{[n]\setminus\{i_1,\ldots,i_b,j_1,\ldots,j_b\}}=Y_{[n]\setminus \{i_1,\ldots,i_{b-1},j_1,\ldots,j_{b-1}\}}\sum_{a=i_b}^{j_b-1}\partial_a.$$
    Applying this identity successively for $b=k,k-1,\ldots,1$ we obtain \begin{align*}\ct \prod_{(i<j)\in \mathcal{M}}\partial_{ij}&=Y_{i_1j_1}\partial_{i_1j_1}Y_{i_2j_2}\partial_{i_2j_2}Y_{i_3j_3}\cdots Y_{i_kj_k}\partial_{i_kj_k}Y_{[n]\setminus \{i_1,\ldots,i_k,j_1,\ldots,j_k\}}\\
    &=Y_{i_1j_1}\partial_{i_1j_1}Y_{i_2j_2}\partial_{i_2j_2}Y_{i_3j_3}\cdots Y_{i_{k-1}j_{k-1}}\partial_{i_{k-1}j_{k-1}}Y_{[n]\setminus \{i_1,\ldots,i_{k-1},j_1,\ldots,j_{k-1}\}}\sum_{a=i_k}^{j_k-1}\partial_a
    \\
    &=\cdots = Y_{[n]}\prod_{(i<j)\in \mathcal{M}}\sum_{a=i}^{j-1}\partial_a=\ct \prod_{(i<j)\in \mathcal{M}}\sum_{a=i}^{j-1}\partial_a\end{align*}
    as desired.
\end{proof}

We note similar reasoning allows us to compute the full pullback map
$\iota^*\colon H^\bullet(\fl{n})\to H^\bullet(\mathcal{B}_{\mathcal{M}})$ on the entire Schubert
basis, not just in degree $k$. Indeed, $H^\bullet(\mathcal{B}_\mathcal{M})$ is
spanned by the squarefree monomials $\prod_{p\in S}y_{i_p}$ for $S\subseteq\{1,\dots,k\}$, and the same combinatorial proofs applied to the (not necessarily packed) submatching $\mathcal{M}_S=\{(i_a,j_a)\suchthat a\in S\}\subset \mathcal{M}$ show
$$[\prod_{p\in S}y_{i_p}]\iota^*f=\langle \prod_{(i<j)\in \mathcal{M}_S}(z_i-z_j),f\rangle=\ct\prod_{(i<j)\in \mathcal{M}_S}\partial_{ij}f=\ct\prod_{p\in S}\sum_{a=i_p}^{j_p-1}\partial_a f.$$
Denoting $\ell(w)$ for the length of $w\in S_n$, this implies
\begin{equation}
\label{eqref:NSw}
\iota^*\schub{w}
=\sum_{S\subseteq\{1,\dots,k\}\text{ and }|S|=\ell(w)}
N_S(w)\,\prod_{p\in S}y_{i_p},
\end{equation}
where $N_S(w)$ is the number of reduced words $s_{a_{p_1}}\cdots s_{a_{p_{\ell(w)}}}$ for
$w$ with $i_p\le a_p<j_p$ for each $p\in S$, the letters taken in the big-to-small
order on $S$. The case $S=\{1,\dots,k\}$ recovers $h^{\mathcal{B}_{\mathcal{M}}}_w$, and in
particular $\iota^*\schub{w}=0$ whenever $\ell(w)>k$.

Recall that $\schub{w}$ also admits a well-known monomial expansion
 over reduced pipe dreams $D$ \cite{BJS93}, and applying
$\iota^*$ termwise expresses $N_S(w)$ as a \textit{signed} sum over those
pipe dreams whose cross tiles lie exclusively in the rows $\bigcup\mathcal{M}_S$.
\begin{question}
   Can \eqref{eqref:NSw} be shown directly from the reduced pipe-dream expansion?
\end{question}

\section{The reverse Artin representative of the  Poincar\'e dual class}

The cohomology ring $H^\bullet(\fl{n})$ has a distinguished \emph{reverse Artin monomial basis}
$$\{x_1^{i_1}\cdots x_{n-1}^{i_{n-1}}\suchthat 0\le i_j\le n-j\}\subset H^\bullet(\fl{n}),$$
the leading terms of the Schubert polynomials $\{\schub{w}\suchthat w\in S_n\}$.
We now verify a conjecture of Precup--Sabando-Alvarez \cite[Conjecture 1]{sabandoalvarez2026idealsdefiningcomponentstworow} on the reverse Artin monomial representative of the Poincar\'e dual class. This means the determination of an element $P_{\mathcal{M}}\in H^\bullet(\fl{n})$ in the span of the reverse Artin monomials such that
$$\int_{\XM}f=\int_{\fl{n}}P_{\mathcal{M}}f.$$
In \cite[Lemma 7.12]{sabandoalvarez2026idealsdefiningcomponentstworow} the general case is reduced to the case that $\mathcal{M}$ is a packed noncrossing matching with \emph{no unmatched vertices}, i.e. $n=2|\mathcal{M}|$, leaving the following conjecture that we now prove.
\begin{thm}[{\cite[Conjecture~1]{sabandoalvarez2026idealsdefiningcomponentstworow}}]\label{thm:conj1}
Let $\mathcal{M}$ be a packed
noncrossing  matching on $[n]$ with $k$ arcs,  where $n=2k$, and write $S(j)$ for the number of
arcs with left endpoint $\le j$. Then
the reverse Artin monomial representative of the Poincar\'e dual class of $[\XM]$ is
\[
  P_{\mathcal{M}}=\prod_{j:\,S(j)<k}x_j^{\,2(k-S(j))}
  \cdot\prod_{(i<j)\in \mathcal{M}}e_{j-i-1}(x_i,\dots,x_{j-1})\in H^\bullet(\fl{n}).
\]
\end{thm}
Since
$[\mathcal{B}_{\mathcal{M}}]=\prod_{(i<j)\in \mathcal{M}}(z_i-z_j)$, by Fact~\ref{fact:flagdel} the statement that $P_{\mathcal{M}}$ represents the Poincar\'e dual class of $[\mathcal{B}_{\mathcal{M}}]$
is equivalent to
\begin{equation}\label{eq:star}
  \ct\partial_{w_{\circ}}(P_{\mathcal{M}}\,f)=\big\langle\prod_{(i<j)\in \mathcal{M}}(z_i-z_j),\,f\big\rangle.
  \tag{$\star$}
\end{equation}
As the reverse Artin monomials form a $\mathbb{Z}$-basis of $H^\bullet(\fl{n})$, the reverse Artin representative of a class is unique, so Theorem~\ref{thm:conj1} is equivalent to \eqref{eq:star} in conjunction with the following count.

\begin{lem}\label{lem:revArtinMembership}
Every monomial occurring in the expansion of $P_{\mathcal{M}}$ satisfies $\deg_{x_t}\le n-t$ for $1\le t\le n$. In particular $P_{\mathcal{M}}$ lies in the span of the reverse Artin monomials.
\end{lem}
\begin{proof}
Fix $t\in [n]$ and count the exponent of $x_t$ contributed by each factor of $P_{\mathcal{M}}$. The first product contributes exactly $2(k-S(t))$, where this is $0$ when $S(t)=k$. Each factor $e_{j-i-1}(x_i,\ldots,x_{j-1})$ is squarefree in its variables, so the arc $(i,j)$ contributes at most $1$, and only when $i\le t<j$. Since $\mathcal{M}$ is a perfect matching of $[n]$, the integers $1,\ldots,t$ comprise $S(t)$ left endpoints and $t-S(t)$ right endpoints of arcs of $\mathcal{M}$, and every arc whose right endpoint is at most $t$ has its left endpoint at most $t$ as well; hence exactly $S(t)-(t-S(t))=2S(t)-t$ arcs $(i,j)$ satisfy $i\le t<j$. Altogether the exponent of $x_t$ in any monomial of $P_{\mathcal{M}}$ is at most
$$2(k-S(t))+\big(2S(t)-t\big)=2k-t=n-t.\qedhere$$
\end{proof}
In light of Lemma~\ref{lem:revArtinMembership} it suffices to establish~\eqref{eq:star}.
We start by considering \emph{connected} matchings, meaning $(1,n)\in \mathcal{M}$.
For $f\in\mathbb{Z}[x_1,\dots,x_n]$ let $f_1=Y_{1n}\frac{d}{dx_1}f\in \mathbb{Z}[x_2,\ldots,x_{n-1}]$ and $f_n=Y_{1n}\frac{d}{dx_n}f\in \mathbb{Z}[x_2,\ldots,x_{n-1}]$ be the coefficients of $x_1^1$ and $x_n^1$ in
the parts of $f$ free of $x_n$, resp.\ $x_1$. Let $\mathcal{M}'$ be the matching on $\{2,\dots,n-1\}$ resulting when we delete the $(1,n)$ arc from $\mathcal{M}$.
Note that
\begin{equation}
\label{eqn:remove1n} \big\langle\prod_{(i<j)\in \mathcal{M}}(z_i-z_j),\,f\big\rangle=\big\langle \prod_{(i<j)\in \mathcal{M}'}(z_i-z_j),\,f_1-f_n\big\rangle,
\end{equation}
where the right-hand pairing can be done in variables $2,\ldots,n-1$.
Note also that the definition of $P_{\mathcal{M}}$ gives
\begin{equation}\label{eqn:MtoM'}
  P_{\mathcal{M}}=x_1^{\,n-2}\,e_{n-2}(x_1,\dots,x_{n-1})\,P_{\mathcal{M}'} .
\end{equation}

\begin{lem}\label{lem:nest}
We have\[
 \ct\partial_{w_{\circ}}\big(x_1^{\,n-2}e_{n-2}(x_1,\dots,x_{n-1})\,f\big)=\ct\partial_{2\cdots(n-1)}\big(f_1-f_n\big).
\]
\end{lem}
\begin{proof}
By linearity it suffices to take $f=x_1^{b_1}\cdots x_n^{b_n}$. Writing $\bar b=(b_2,\dots,b_{n-1})$ and $\sgn(\bar
b)$ the sign of the permutation sorting $\bar b$ into decreasing order, we have by Fact~\ref{fact:flagdel} that

\begin{align}\label{eq:summary_rhs}
\ct\partial_{w_{\circ}}\big(f_1-f_n\big)=
\begin{cases}
\sgn(\bar b) & \text{if } b_1=1,\ b_n=0, \{b_2,\dots,b_{n-1}\}=\{0,\dots,n-3\}\\[2pt]
-\sgn({\bar b}) & \text{if } b_1=0,\ b_n=1, \{b_2,\dots,b_{n-1}\}=\{0,\dots,n-3\}\\[2pt]
0 & \text{otherwise.}
\end{cases}
\end{align}
    For $q$ an indeterminate the $q^{n-2}$-coefficient of the identity
    $(1+qx_1)\cdots(1+qx_{n-1})=(1+qx_1)\cdots (1+qx_n)(1+qx_n)^{-1}$ shows that
    $$e_{n-2}(x_1,\ldots,x_{n-1})=\sum_{j=0}^{n-2}(-1)^je_{n-2-j}(x_1,\ldots,x_n)x_n^j,$$
    so by \eqref{eqn:pullsymthrough} we have
    $$\ct\partial_{w_\circ}(x_1^{n-2}e_{n-2}(x_1,\ldots,x_{n-1})x_1^{a_1}\cdots x_n^{a_n})=(-1)^{n-2}\ct\partial_{w_{\circ}}((x_1^{n-2}x_{n}^{n-2})x_1^{a_1}\cdots x_n^{a_n}).$$
    By Fact~\ref{fact:flagdel} this is nonzero if and only if $(n-2+a_1,a_2,a_3,\ldots,a_{n-1},n-2+a_n)=(n-\sigma(1),n-\sigma(2),\ldots,n-\sigma(n))$ for some $\sigma\in S_n$, in which case it is $(-1)^{n-2+\ell(\sigma)}$. This forces $(a_1,a_n)=(1,0)$ or $(0,1)$ and the identification with \eqref{eq:summary_rhs} immediately follows.
\end{proof}
\begin{proof}[{Proof of Theorem~\ref{thm:conj1}}]
We now establish~\eqref{eq:star}.
Induct on $k$.
For $k=1$, $\mathcal{M}=\{(1,2)\}$, $P_{\mathcal{M}}=1$, and \eqref{eq:star} is
Lemma~\ref{lem:nest}.
Now suppose $k>1$. If $\mathcal{M}$ is  composed of connected matchings $\mathcal{M}_1,\mathcal{M}_2,\ldots,\mathcal{M}_r$ on disjoint intervals $I_1,\ldots,I_r\subset [n]$ then \cite[Lemma 7.12]{sabandoalvarez2026idealsdefiningcomponentstworow} shows that the result for $\mathcal{M}_1,\ldots,\mathcal{M}_r$ implies the result for $\mathcal{M}$. It remains to verify the case that $\mathcal{M}$ is connected. As before we write $\mathcal{M}=\{(1,n)\}\sqcup \mathcal{M}'$.
By  \eqref{eqn:MtoM'}, Lemma~\ref{lem:nest}, the inductive hypothesis for $\mathcal{M}'$, and \eqref{eqn:remove1n}, we deduce
\begin{align}\label{eq:reduced}
  \ct\partial_{w_{\circ}}(P_{\mathcal{M}} f)=\ct\partial_{2\cdots(n-1)}\big(P_{\mathcal{M}'}\,(f_1-f_n)\big)&=\big\langle\textstyle\prod_{(i<j)\in\mathcal{M}'}(z_i-z_j),\,f_1-f_n\big\rangle\nonumber\\&=\big\langle \prod_{(i<j)\in \mathcal{M}}(z_i-z_j),f\big\rangle.\qedhere
\end{align}
\end{proof}

\section{Further varieties indexed by matchings}
\label{sec:FurtherMatchings}

We conclude by relating our varieties to two families indexed by matchings that have appeared in the literature: the Poisson-geometric Richardson varieties of Casbi--Masoomi--Yakimov~\cite{CMY}, and certain torus-orbit closures inside the quasisymmetric flag variety $\operatorname{QFL}_n$ of \cite{BGNST2}. In each case we deduce a combinatorially nonnegative Schubert cycle decomposition.

\smallskip
\noindent\textit{Poisson degeneracy loci.}
For a connected semisimple group~$G$, Casbi--Masoomi--Yakimov \cite{CMY} describe the strata of the
Poisson degeneracy locus of $G/B$ as the union of the Richardson varieties $\mathcal{R}_{v,w}^{CMY}$ whose
defining pair satisfies $v\tau_{\gamma_1}\cdots \tau_{\gamma_k} = w$ for pairwise orthogonal
positive roots $\gamma_1,\dots,\gamma_k$ with $k=\ell(w)-\ell(v)$
\cite[Thm.~A]{CMY}.
They established that the corresponding closed Richardson variety
is isomorphic to
$\mathcal{R}_{v,w}^{CMY}\cong(\mathbb{CP}^1)^k$ with Bruhat interval $[v,w]$ isomorphic
to the rank $k$ Boolean lattice \cite[Thm.~B]{CMY}.
The intermediate elements of $[v,w]$ correspond to subproducts $v\prod_{i\in S}\tau_{\gamma_i}$ for $S\subset [k]$.

In type~$A$ the
pairwise orthogonality of the roots is exactly the disjointness of the pairs
$(i_a,j_a)$ associated to $\gamma_a$, so they form a matching $\mathcal{M}$. In the special case $$\{w\tau_{\gamma_1}w^{-1},\ldots,w\tau_{\gamma_k}w^{-1}\}=\{(1,2),(3,4),\ldots,(2k-1,2k)\},$$ we recover the Richardsons considered above.

Let $\operatorname{Diag}_n\coloneqq B\cap B^-\subset \GL_n$ denote the torus of diagonal matrices. To see that $\mathcal{R}_{v,w}^{CMY}$ arise as translates of Levi-Richardsons $\mathcal{R}_\mathcal{C}$ associated to two-row partitions $\lambda=(n-k,k)$ in general, note that if $\mathcal{R}_{v,w}^{CMY}$ and $\mathcal{R}_{v',w'}^{CMY}$ share the same matching then $(v'v^{-1})\mathcal{R}_{v,w}^{CMY}=\mathcal{R}_{v',w'}^{CMY}$ as they share the same $\operatorname{Diag}_n$-fixed points (i.e. permutation matrices) and Richardson varieties are $\operatorname{Diag}_n$-convex varieties with respect to the diagonal  torus \cite[Theorem 6.3]{BCP25}. In particular, if $\mathcal{M}$ is the matching induced by the supports of the reflections $\tau_{\gamma_1},\ldots,\tau_{\gamma_k}$, then by Theorem~\ref{thm:RCcohomology} we have
$$[\mathcal{R}_{v,w}^{CMY}]=[\mathcal{R}_{\mathcal{C}}]=\prod_{(i<j)\in \mathcal{M}}(z_i-z_j)\in H_\bullet(\fl{n})$$
for any $\mathcal{C}\in S_{\lambda^{\top}}\setminus S_n$ whose underlying matching is $\mathcal{M}$, and Corollary~\ref{cor:RCtoSchub} then gives a combinatorially nonnegative Schubert cycle decomposition.

\smallskip
\noindent\textit{Torus-orbit closures in the quasisymmetric flag variety.}
    Now consider again for $\lambda=(n-k,k)$ the Levi-Richardson varieties $\mathcal{R}_{\mathcal{C}}$ where the associated matching for $\mathcal{C}$ is noncrossing (not necessarily packed). Then $(\min \mathcal{C})^{-1}\mathcal{R}_{\mathcal{C}}$ contains exactly those $\operatorname{Diag}_n$-fixed points associated to a subproduct of the transpositions of the noncrossing matching. This is a special case of an algebraic noncrossing partition, and therefore $(\min \mathcal{C})^{-1}\mathcal{R}_{\mathcal{C}}\subset \operatorname{QFL}_n$, the quasisymmetric flag variety, developed by the present authors together with Bergeron, Gagnon, Nadeau in \cite{BGNST2}.
    In the nested bicoloured forest encoding of \textit{ibid.}, each arc $(i,j)\in\mathcal{M}$
contributes a single-node black tree supported on $\{i,j\}$, and this is the torus-orbit
closure $X(\mathcal{M})$ attached to this nested black forest. Its
combinatorially positive Schubert cycle expansion follows either from the results of the
present paper or from prior work on Schubert cycle expansions of torus-orbit closures in
$\operatorname{QFL}_n$ \cite{NST_c,BGNST2}.

\bibliographystyle{hplain}
\bibliography{qfl.bib}

\end{document}